\documentclass[12pt]{article}
\usepackage{latexsym}
\usepackage{amsmath,amsthm,amssymb}
\usepackage{graphicx}
\theoremstyle{remark}
\newtheorem*{remark}{Remark}
\newtheorem*{convention*}{Convention}
\theoremstyle{plain}
\newtheorem{theorem}{Theorem}[section]
\newtheorem{lemma}[theorem]{Lemma}
\newtheorem{proposition}[theorem]{Proposition}

\theoremstyle{definition}

\newcommand{\bR}{{\mathbf R}}

\newcommand{\com}{\circ}
\DeclareMathOperator{\cl}{cl}
\providecommand{\abs[1]}{\lvert#1\rvert}
\providecommand{\bigabs[1]}{\bigl\lvert#1\bigr\rvert}
\providecommand{\Bigabs[1]}{\Bigl \lvert#1\Bigr \rvert}
\providecommand{\biggabs[1]}{\biggl \lvert#1\biggr \rvert}
\providecommand{\norm[1]}{\lVert#1\rVert}
\providecommand{\bignorm[1]}{\bigl\lVert #1 \bigr\rVert}
\providecommand{\Dist[2]}{{\rm dist}\bigl( #1, #2 \bigr)}
\begin{document}
\title{A Nutrient-Prey-Predator Model:  Stability and Bifurcations}%[Bifurcation in a Three Component System]
\author{Mary Ballyk
 \and Ibrahim Jawarneh 
\\ \\
Department of Mathematical Sciences
\\
New Mexico State University
\\
Las Cruces, NM 88003, USA
\and Ross Staffeldt}
\maketitle
\begin{abstract}
In this paper we consider a model of a nutrient-prey-predator
system in a chemostat with general functional responses, using the 
input concentration of nutrient as the bifurcation parameter. 
We study the changes  in the existence of isolated equilibria and in their stability, 
as well as the global dynamics, as the nutrient concentration varies.
The bifurcations of the system are analytically verified  and we identify conditions under which 
an equilibrium undergoes a Hopf bifurcation and a limit cycle appears. 
Numerical simulations for specific functional responses illustrate the general results.
\end{abstract}
%92D25 View Publications (1991-now) Population dynamics (general) ;34A34 View Publications (1980-now) Nonlinear equations and systems, general ;
%34C23 View Publications (1991-now) Bifurcation [See also 37Gxx] 37G10 View Publications (2000-now) Bifurcations of singular points 
%Introduction
\section{Introduction}
\label{Introduction}
 We consider a mathematical model of two-species predator-prey interaction in the chemostat under nutrient limitation.
 With the exception of one nutrient, all nutrients that the prey species
 requires are supplied to the growth vessel 
from the feed vessel in ample supply. The predator species grows exclusively on the prey.  
 With $N$ the concentration of the limiting nutrient,  
$P$ the concentration of prey (say, phytoplankton),
and $Z$ the concentration of predator (say, zooplankton), we consider the following model:
\begin{align} 
  dN/dt &= (\mu - N) D - P f_1(N)  \notag
\\
 dP/dt  &= \gamma_1 P f_1(N) - D_1 P - Z f_2(P)  \label{NPZsys}
\\ 
 dZ/dt  &= \gamma_2 Z f_2(P) - D_2 Z  \notag
\end{align}
for initial conditions $N(0) = N_0 >0$, $P(0) = P_0  \geq 0$, and $Z(0) = Z_0 \geq 0$. 

The concentration of the growth-limiting nutrient in the feed vessel is
denoted $\mu$, and will be the bifurcation parameter in our analysis.
$D$ is the input rate from the feed vessel to the growth vessel as well as the washout rate from the growth vessel to the receptacle, 
so that constant volume is maintained. 
The parameters $D_{1}=D+\epsilon_1$ and $D_{2}=D+\epsilon_2$ are the removal rates of $P$ and $Z$, respectively, from the growth vessel, 
incorporating the washout rate $D$ and the intrinsic death rates $\epsilon_i$ of $P$ and $Z$. 
Our analysis does not necessarily require that $\epsilon_1$ and $\epsilon_2$ are positive; however, $D_1$ and $D_2$ should be positive.
The yield coefficient $\gamma_1$ gives the amount of prey biomass produced per unit of nutrient consumed, 
while $\gamma_2$ gives the amount of predator biomass produced per unit of prey biomass consumed.

The function $f_1(N)$ represents the per capita consumption rate of nutrient by the prey populations
 as a function of the concentration of available nutrient; 
similarly, the function $f_2(P)$ represents the per capita consumption rate of the prey by the predator
 as a function of available prey. These functions are assumed to satisfy 
$f_i(0) = 0$, $i=1$, 2, $f_1'(N)>0$ for all $N\geq 0$, and $f_2'(P)>0$ for all $P\geq 0$. 
We further assume that $f_1(N)$ and $f_2(P)$ are bounded. 
To avoid the case of washout due to an inadequate resource, we assume that 
$\lim_{N\to\infty} f_{1}(N) > D_1/\gamma_1$, and 
to avoid the case of an inadequate prey, we assume that $\lim_{P\to\infty} f_{2}(P) > D_2/\gamma_2$.   
Define $\lambda_P(D_1)$ and $\lambda_Z(D_2)$ to be the unique numbers satisfying 
\begin{equation}
\label{breakeven}
    f_1\bigl(\lambda_P(D_1)\bigr)=D_1/\gamma_1 \mbox{  and   }
    f_2\bigl(\lambda_Z(D_2)\bigr)=D_2/\gamma_2.
\end{equation}
The number $\lambda_P(D_1)$ is the break-even concentration of nutrient, at which the growth and removal of phytoplankton balance; 
$\lambda_Z(D_2)$ is similarly interpreted.  
The number $D$ plays a central role in our investigation, so we assume that $\lambda_P(D)$ and $\lambda_Z(D)$ are also 
defined.  From the perspective of $D$, due to the boundedness assumptions on $f_1$ and $f_2$, $D_1$ and $D_2$ are
perturbations of $D$. 
\begin{lemma} \label{lambda_diff}
  From a functional point of view, $\lambda_P$ and $\lambda_Z$ are right
  inverses of $\gamma_1f_1$ and $\gamma_2f_2$, respectively.  Accordingly, on their respective domains
$\lambda_P$ and $\lambda_Z$ are as differentiable as $f_1$ and $f_2$.  We note for later use
\begin{equation} \label{lambdaderivs}
  \lambda_P'(D_1) = \bigl( \gamma_1\cdot f_1'(\lambda_P(D_1))\bigr)^{-1}
 \quad \text{and} \quad
 \lambda_Z'(D_2) = \bigl( \gamma_2\cdot f_2'(\lambda_P(D_2))\bigr)^{-1}.  \qed
\end{equation} 
\end{lemma}
Kuang and Li \cite{Kuang1999} studied this system with general functional responses and distinct values of 
$D$, $D_1$, and $D_2$.  However, they fixed the input nutrient concentration, whereas we have this as a parameter.
With the hypothesis that $D{=}D_1{=}D_2$, they provide numerical criteria for the stability of a coexistence equilibrium,
and prove that a cycle exists when the equilibrium is unstable \cite[Theorem~3.2]{Kuang1999}.
When the hypothesis $D{=}D_1{=}D_2$ does not hold, they provide numerical evidence 
that stability of the coexistence equilibrium breaks down and a cycle appears.
A similar model was studied in  \cite{Zhang2012} with functional response $f_1(N)$ of Holling type I and $f_2(P)$ of Holling type II,
demonstrating the existence of a Hopf bifurcation in response to varying nutrient concentration. 
These results inspire our work.
Our goal is to prove analytically that the system undergoes a Hopf bifurcation without restricting the forms of the
uptake functions or the values of the removal rates $D_1$ and $D_2$. 

This paper is organized as follows. 
In section \ref{Equilibria}, 
conditions for existence and local stability of predator-free equilibria are obtained for general functional responses $f_1(N)$ and $f_2(P)$.
In section \ref{Coexistence_equilibrium_stability}, we study stability of a coexistence equilibrium
which appears as the parameter $\mu$ increases.
We  quote a version of the Hopf bifurcation theorem, stating the result in a limited form most useful for us.
Accordingly, application of this theorem requires us to develop  specific information about the behavior of eigenvalues of
the linearizations as the paramter $\mu$ increases. 
In section \ref{Cycles}, we use a smooth change of variables to reach a situation where the Hopf bifurcation theorem in three 
dimensions applies, enabling us to conclude the existence of cycles.
In section \ref{Examples}, our results are illustrated using simulations arising by choosing rate functions of
Holling type II and Holling type III forms. 
In section \ref{Appendix}, we explain in detail the approximations and estimates that support our work 
in the latter part of section \ref{Coexistence_equilibrium_stability}.
\section{Steady States and  Their Stability}
\label{Equilibria}
To begin, we establish that the solutions of \eqref{NPZsys} are  nonnegative and bounded.
These are minimum requirements for a reasonable model of the chemostat. We
then develop conditions for the existence and stability of equilibria.
We conclude the section by proving uniform persistence in the sense of 
\cite{Butler1986} when $\mu$ is sufficiently large and the initial values $P_0$ and $Z_0$ are positive.
\begin{lemma} \label{positivity_boundedness}
All solutions $N(t)$, $P(t)$, and $Z(t)$ of \eqref{NPZsys} are  nonnegative and bounded.
\end{lemma}
\begin{proof}
The plane $Z = 0$ is  invariant for system \eqref{NPZsys}. 
Therefore, by existence and uniqueness, if $Z_0 > 0$ then $Z(t) > 0$ for all $t \geq 0$. 
Similarly, since $f_2(0)=0$, the plane $P = 0$ is invariant, so $P_0 > 0$ implies $P(t) > 0$ for all  $t \geq 0$.

Suppose $N_0 > 0$.  
If there exists a $t > 0$ with $N(t) = 0$, then there is a least such number, say $t_0$. 
Then $N'(t_0) = \mu D > 0$ since $f_1(0)=0$. 
Consequently, there is $t<t_0$ such that $N(t) < 0$, which is a contradiction to the choice of $t_0$.

For the boundedness of solutions, set $U(t) = N(t) + \gamma_1^{-1}P(t) +(\gamma_1 \gamma_2)^{-1}Z(t)$.
From system \eqref{NPZsys}, it follows that 
\begin{equation*}
U'(t) = D \mu - D N(t) -\gamma_1^{-1} D_1 P(t) 
           -(\gamma_1 \gamma_2)^{-1}D_2 Z(t) \leq  D \mu - \widehat{D} U(t),
\end{equation*}
where $\widehat{D} = \min\{D,D_1,D_2\}$. Then 
\begin{align*} 
  U(t) &\leq (D\mu)/{\widehat{D}} + \bigl(U(0) - (D\mu)/\widehat{D} \bigr) e^{-\widehat{D}t} 
  \\
  &\leq
    \begin{cases}  U(0),  &   \text{if  $U(0) > (D\mu)/\widehat{D}$,} \notag
    \\
    (D\mu)/\widehat{D}, &  \text{if   $U(0) \leq  (D\mu)/\widehat{D}$.}  
    \end{cases} 
\end{align*} 
Since $N(t)$, $P(t)$, and $Z(t)$ are nonnegative, 
the boundedness of $U(t)$ implies the boundedness of  $N(t)$, $P(t)$, and $Z(t)$.
\end{proof}

There are at most three  biologically relevant equilibria of 
system \eqref{NPZsys} depending on the value of $\mu$. 
The equilibria and the conditions of their existence are summarized in the following theorem.
\begin{theorem}  \label{existence equilibria NPZ}
Let $\lambda_P(D_1)$ and $\lambda_Z(D_2)$ be as in \eqref{breakeven}.
 The equilibria of the system \eqref{NPZsys} satisfy  the following conditions:
\begin{enumerate}
    \item The washout equilibrium $E_{0} = (\mu,0,0)$ exists for all $\mu > 0$.    
    \item The single-species equilibrium $E_1(\mu, D_1) = ( \lambda_P(D_1), P(\mu, D_1), 0)$ exists for all $\mu > \lambda_P(D_1)$, 
 where 
 \begin{equation}
   \label{Poneformula}
    P(\mu, D_1) = \bigl(\mu - \lambda_P(D_1)\bigr) \frac{D \gamma_1}{D_1}.
 \end{equation}
    \item The coexistence equilibrium 
      \begin{equation*}
E_2(\mu, D_1, D_2) = \bigl(N(\mu, D_1, D_2), \lambda_Z(D_2), Z(\mu, D_1, D_2)\bigr)        
      \end{equation*}
exists for all $\displaystyle{\mu > \mu_{ c_1}(D_1,D_2) } $, 
where 
    \begin{equation}\label{mu_{c1}}
    \mu_{c_1}(D_1,D_2)= \lambda_P(D_1) + \frac{D_1 \lambda_Z(D_2)}{D\gamma_1}
    \end{equation}
and 
$N(\mu, D_1, D_2)$, $Z(\mu, D_1, D_2)$ satisfy the simultaneous equations 
\begin{gather}
 \bigl(\mu - N(\mu, D_1, D_2)\bigr) D - \lambda_Z(D_2) f_1\bigl(N(\mu, D_1, D_2)\bigr) = 0,  \label{1steqnwithZ_2notzero}
\\
 \gamma_1 \lambda_Z(D_2) f_1\bigl(N(\mu, D_1, D_2)\bigr) - D_1 \lambda_Z(D_2) 
          - Z(\mu, D_1, D_2) f_2\bigl(\lambda_Z(D_2)\bigr) = 0. \label{2ndeqnwithZ_2notzero}
 \end{gather}
\end{enumerate}
\end{theorem}
Thus, for $\mu \leq \lambda_P(D_1)$ only the equilibrium $E_0$ exists; for
$\lambda_P(D_1) < \mu \leq \mu_{c_1}(D_1,D_2)$, there are two equilibria
$E_0, E_1$; and, when $ \mu_{c_1}(D_1,D_2) < \mu$, there are three equilibria $E_0$, $E_1$, and $E_2$.
\begin{proof}
From the $Z$-equation, 
either $Z = 0$ or $\gamma_2 f_2(P) - D_2 = 0$ (so that $P=\lambda_Z(D_2)$). 
If $Z = 0$, the $P$-equation yields  
\begin{equation*}
   0 = \gamma_1 P f_1(N) - D_1 P,
\end{equation*}
which implies either $P = 0$ or $\gamma_1 f_1(N) - D_1 =0$ (so that $N=\lambda_P(D_1)$).  

If $Z=0$ and $P=0$, then the $N$-equation gives 
$0= (\mu - N) D $, so that $N = \mu$. Thus, the washout equilibrium is given by $E_0 = (\mu, 0, 0)$ 
and exists for $\mu > 0$. 
Note that as $\mu$ increases, $E_0$ moves along the $N$-axis in  $\bR_{\geq 0}^{3}$. 
This is the proof of the first part of the theorem.

If $Z = 0$ and $P \neq 0$, then $N = \lambda_P(D_1)$ in the $N$-equation gives 
\begin{equation} \label{P_1 equation}
  0 = \bigl(\mu - \lambda_P(D_1)\bigr) D - P \frac{D_1}{\gamma_1},
\quad \text{with solution} \quad
P = \bigl(\mu - \lambda_P(D_1)\bigr) \frac{D \gamma_1}{D_1}\mathrel{\mathop :}= P(\mu, D_1).
\end{equation}
Note that $P(\mu, D_1)>0$ for all $\mu>\lambda_P(D_1)$, 
and so the single-species equilibrium 
$E_{1}(\mu, D_1) = (\lambda_P(D_1), P(\mu, D_1),0)$ 
exists in the positive cone for all $\mu > \lambda_P(D_1)$.
This is  the proof of the second part of the theorem.

If $Z \neq 0$, then $P = \lambda_Z(D_2)$ in the $N$- and $P$-equations, and
equations \eqref{1steqnwithZ_2notzero}  and \eqref{2ndeqnwithZ_2notzero}
 define $N$ and $Z$ as implicit functions of $\mu$, $D_1$, and $D_2$.
 
 We now determine the critical value $\mu_{c_1}(D_1,D_2)$ of $\mu$ at which 
$E_2(\mu, D_1, D_2)$
first appears in the positive cone. 
 Equation \eqref{2ndeqnwithZ_2notzero} implies
\begin{equation} \label{valueZ_2}
Z(\mu, D_1, D_2) = \frac{\gamma_2}{D_2}\lambda_Z(D_2)\bigl(\gamma_1f_1\bigl(N(\mu, D_1, D_2)\bigr) - D_1 \bigr). 
\end{equation} 
Since  $f_1$ is strictly increasing, $Z(\mu, D_1, D_2) =  0$ if and only if $N(\mu, D_1, D_2) = \lambda_P(D_1)$,
and $Z(\mu, D_1, D_2) >  0$ for $N(\mu, D_1, D_2) > \lambda_P(D_1)$.
Substituting $N(\mu, D_1, D_2) = \lambda_P(D_1)$  in \eqref{1steqnwithZ_2notzero}, we obtain the equation
\begin{equation*} 
    0 = \bigl( \mu - \lambda_P(D_1)\bigr) D - \lambda_Z(D_2) \frac{D_1}{\gamma_1} 
\; \text{with solution} \; 
\mu = \lambda_P(D_1) + \frac{D_1 \lambda_Z(D_2)}{D\gamma_1}  \mathrel{\mathop :}= \mu_{c_1}(D_1,D_2).
\end{equation*}
Thus, $Z(\mu, D_1, D_2)$ is positive when $\mu > \mu_{c_1}(D_1, D_2)$,
and the coexistence equilibrium 
$E_2 (\mu, D_1, D_2)= \bigl(N(\mu, D_1, D_2), \lambda_Z(D_2), Z(\mu, D_1, D_2)\bigr)$  
exists in the positive cone for all $\mu > \mu_{c_1}(D_1, D_2)$. 
This proves part  three of the theorem.
\end{proof} 
In Theorems~\ref{NPZ stability E_0}, \ref{NPZ stability of E_1}, and \ref{coexistence_equilibrium_1} 
we investigate the stability of the equilibria of system \eqref{NPZsys} 
by finding the eigenvalues of the associated Jacobian matrices. 
The Jacobian matrix of the system \eqref{NPZsys} takes the form
\begin{equation} \label{Jacobian}
J =
  \begin{bmatrix}
  -D & -f_{1}(N) & 0\\
\gamma_{1}P f_1'(N) & \gamma_{1}f_{1}(N)-D_{1}-Z f_2'(P) & -f_{2}(P)   \\
0 & \gamma_{2}Z f_2'(P) & \gamma_{2}f_{2}(P) -D_{2}
\end{bmatrix}  .
\end{equation}
We first summarize the stability of $E_0$ in the following theorem. 
Here, the breakeven  concentration of nutrient given in \eqref{breakeven} plays a critical role.
\begin{theorem} \label{NPZ stability E_0}
 The equilibrium point  $E_{0}$ is locally asymptotically stable 
if $\mu < \lambda_P(D_1)$ and unstable if $\mu > \lambda_P(D_1)$.  
When $\mu > \lambda_P(D_1)$, $E_0$ is globally asymptotically stable 
with respect to solutions initiating in $\{(N,P,Z) \in \bR^3_+ \mid P = 0 \}$.
That is, the plane $P=0$ is $m^+(E_0)$, the stable manifold of $E_0$.
\end{theorem}
\begin{proof}
The Jacobian at $E_{0}=(\mu, 0 ,0)$ is
\begin{equation*}
J(E_{0}) =
  \begin{bmatrix}
  -D & -f_{1}(\mu) & 0\\
0 & \gamma_{1}f_{1}(\mu)-D_{1}& 0   \\
0 & 0 & -D_{2}
\end{bmatrix},
\end{equation*}
so that the eigenvalues of $J(E_0)$ are $x_1 = -D, x_2 = \gamma_1 f_1(\mu) - D_1$, 
and  $\displaystyle{x_3 = -D_2 }$. The stability of $E_0$ now follows from \eqref{breakeven} and the fact that $f_1$ is strictly increasing: 
$x_2<0$ when $\mu<\lambda_P(D_1)$ and $x_2>0$ when $\mu>\lambda_P(D_1)$. 
To see that $m^+(E_0)= \{(N,P,Z) \in \bR^3_+ \mid P = 0 \}$ when $\mu>\lambda_P(D_1)$, 
consider the Lyapunov function
\begin{equation*}
    L(N,Z) = \frac{(\mu -N)^2}{2}  +  \frac{Z^2}{2}.
\end{equation*}
Clearly, $L(\mu,0) = 0$ and $L(N,Z) > 0$ if $(N,Z) \neq (\mu,0) $. 
 The time derivative of $L(N,Z)$ at a point $(N,0,Z)$ on a trajectory  of system \eqref{NPZsys} is
 \begin{equation*}
   L'(N,Z) = -D(\mu-N)^2-D_2Z^2 < 0 
 \end{equation*}
for $(N,Z) \ne (\mu,0)$. 
Thus $E_0$ is globally asymptotically stable in the plane $P = 0$. 
\end{proof} 

For $\mu=\lambda_P(D_1)$, $P(\lambda_P(D_1), D_1)=0$, so that $E_0$ and $E_1$ coalesce (see equation \eqref{Poneformula}).
When $\mu > \lambda_P(D_1) $, $E_1(\mu,D_1) = (\lambda_{P}(D_1),P(\mu, D_1),0)$ enters the positive cone. 
We summarize the stability of $E_{1}(\mu, D_1)$ in the following theorem. 
Note that the critical value of $\mu$ given in \eqref{mu_{c1}} now plays a central role.
\begin{theorem} \label{NPZ stability of E_1}
The equilibrium point  $E_1(\mu, D_1)$ is locally stable if $\lambda_P(D_1)< \mu < \mu_{c_1}(D_1,D_2)$ and unstable 
if $\mu > \mu_{c_1}(D_1,D_2)$.  
When $\mu > \mu_{c_1}(D_1,D_2)$,  $E_1(\mu, D_1)$ is globally asymptotically stable 
with respect to solutions initiating in $\{(N,P,Z) \in \bR^3_+ \mid Z = 0 \}$.
That is, the plane $Z=0$ is  $m^+\bigl(E_1(\mu, D_1)\bigr)$,  the stable manifold of $E_1(\mu, D_1)$.
\end{theorem} 
\begin{proof} 
The Jacobian matrix at $E_1(\mu, D_1)$ is
\begin{equation*} 
J\bigl(E_{1}(\mu, D_1)\bigr) =
\\
  \begin{bmatrix}
  -D-P(\mu, D_1) f_1'\bigl(\lambda_{P}(D_1)\bigr) & -f_{1}\bigl(\lambda_{P}(D_1)\bigr) & 0\\
\gamma_{1}P(\mu, D_1) f_1'\bigl(\lambda_{P}(D_1)\bigr) & 0& -f_{2}\bigl(P(\mu, D_1)\bigr))  \\
0 & 0 & \gamma_{2}f_{2}\bigl(P(\mu, D_1)\bigr)-D_{2}
\end{bmatrix}.  
\end{equation*}
The determinant of the upper lefthand $2$-by-$2$ submatrix is positive and its trace is negative, 
so its eigenvalues have negative real parts. The third eigenvalue is 
$x_3 = \gamma_{2}f_{2}\bigl(P\bigl(\mu, D_1)\bigr)-D_{2}$. 
If   $\mu < \mu_{c_1}(D_1,D_2)$, so that $P(\mu, D_1) < \lambda_{Z}(D_2)$, then $x_3<0$, and $E_1(\mu, D_1)$ is locally stable. 
Similarly, if $\mu >\mu_{c_1}(D_1,D_2)$, so that $P(\mu, D_1) > \lambda_Z(D_2)$, 
then $x_3>0$ and $E_1(\mu, D_1)$ is unstable with one dimension of instability.

To see that $m^+(E_1)=\{(N,P,Z) \in \bR^3_+ \mid Z = 0 \}$ when $\mu >
\mu_{c_1}$, consider the  Lyapunov function 
introduced by Hsu in \cite{Hsu1978}
\begin{multline*}
  L(N,P) = 
\\
\int_{\lambda_P(D_1)}^N \frac{ f_1(n) -f_1\bigl(\lambda_P(D_1)\bigr)}{f_1(n)} \, dn
+ 
\frac{1}{\gamma_1}\Bigl( P - P(\mu, D_1) - P(\mu, D_1) \ln \bigl(P/P(\mu, D_1)\bigr) \Bigr).
\end{multline*}
Notice that $L\bigl(\lambda_P(D_1), P(\mu, D_1)\bigr) = 0$, 
\begin{equation*}
  \frac{\partial L}{\partial N} = \frac{f_1(N) -f_1\bigl(\lambda_P(D_1)\bigr)}{f_1(N)} = 0
\quad \text{and} \quad
  \frac{\partial L}{\partial P} = \frac{1}{\gamma_1}\bigl( 1 - P(\mu, D_1)/P \bigr) = 0
\end{equation*}
precisely when $(N,P) = \bigl(\lambda_P(D_1),P(\mu, D_1)\bigr)$.  Moreover,
\begin{align*}
  \frac{\partial^2 L}{\partial N^2}\bigl(\lambda_P(D_1), P(\mu, D_1)\bigr) &= 
                   \frac{f'_1\bigl(\lambda_P(D_1)\bigr)}{f_1\bigl(\lambda_P(D_1)\bigr)}>0
\intertext{and}
  \frac{\partial^2 L}{\partial P^2}\bigl(\lambda_P(D_1), P(\mu, D_1)\bigr) &= \frac{1}{\gamma_1 P(\mu,D_1)}>0.
\end{align*}
Therefore, $\bigl(\lambda_P(D_1), P(\mu, D_1)\bigr)$ is the only critical point of $L(P,N)$ and it is a local minimum, so 
that $L(N,P) > 0$ for all $(N,P) \neq \bigl(\lambda_P(D_1), P(\mu, D_1)\bigr)$. 

Now we compute the time derivative of $L(N,P)$ at a point $(N,P,0)$ along a trajectory of system 
\eqref{NPZsys}. 
Noting from \eqref{breakeven} and \eqref{Poneformula} that 
\begin{equation*}
 f_1\bigl(\lambda_P(D_1)\bigr)=\frac{D_1}{\gamma_1} 
\quad \text{and} \quad 
P(\mu,D_1) = \frac{\bigl(\mu-\lambda_P(D_1)\bigr)D}{f_1\bigl(\lambda_P(D_1)\bigr)},  
\end{equation*}
we have
\begin{equation*} 
  \begin{split}
    L'(N,P) &=\frac{ f_1(N) -f_1\bigl(\lambda_P(D_1)\bigr) }{f_1(N)}\bigl((\mu-N)D - P f_1(N) \bigr)
\\
            &  \quad \quad \quad \quad    +  \frac{1}{\gamma_1}\biggl( 1 - \frac{P(\mu, D_1)}{P} \biggr) (\gamma_1 f_1(N) - D_1)   P 
\\
&= \bigl(f_1(N)-f_1\bigl(\lambda_P(D_1)\bigr)\bigr) 
        \biggl(  \frac{(\mu-N)D}{f_1(N)}  -  \frac{\bigl(\mu-\lambda_P(D_1)\bigr)D}{f_1\bigl(\lambda_P(D_1)\bigr)}\biggr)
\\
&= \bigl(f_1(N)-f_1\bigl(\lambda_P(D_1)\bigr)\bigr) \biggl( \frac{(\mu-N)}{\bigl(\mu-\lambda_P(D_1)\bigr)} 
           -  \frac{f_1(N)}{f_1\bigl(\lambda_P(D_1)\bigr)}\biggr) 
           \frac{\bigl(\mu - \lambda_P(D_1)\bigr) D}{f_1(N)}.
\end{split}
\end{equation*}
If $N < \lambda_P(D_1)$, then $f_1(N)-f_1\bigl(\lambda_P(D_1)\bigr) < 0$, so that $f_1(N)/f_1\bigl(\lambda_P(D_1)\bigr) < 1$.
Also, $\mu{-}N > \mu{-}\lambda_P(D_1)$, so that
\begin{equation*}
\frac{\mu {-} N}{\mu {-} \lambda_P(D_1)} > 1 > \frac{f_1(N)}{f_1\bigl(\lambda_P(D_1)\bigr)},\;\text{and}\;
\frac{\mu {-} N}{\mu {-} \lambda_P(D_1)} - \frac{f_1(N)}{f_1\bigl(\lambda_P(D_1)\bigr)} > 0.
\end{equation*}
Thus, $L^{\prime}(N,P) < 0$ when $N < \lambda_P(D_1)$.

If $N>\lambda_P(D_1)$, then $f_1(N)-f_1\bigl(\lambda_P(D_1)\bigr)>0$, so that $f_1(N)/f_1\bigl(\lambda_P(D_1)\bigr)>1$.
When $\mu > N > \lambda_P(D_1)$, we have $\mu {-} N < \mu {-} \lambda_P(D_1)$, so that
\begin{equation*}
\frac{\mu - N}{\mu - \lambda_P(D_1)} < 1 < \frac{f_1(N)}{f_1\bigl(\lambda_P(D_1)\bigr)}, 
\end{equation*}
while for $N \ge \mu$, then $\displaystyle{\frac{\mu - N}{\mu - \lambda_P(D_1)} < 0}$. In either case, 
$\displaystyle{\frac{\mu - N}{\mu - \lambda_P(D_1)} - \frac{f_1(N)}{f_1\bigl(\lambda_P(D_1)\bigr)} < 0}$, and so 
$L^{\prime}(N,P) < 0$ when $N>\lambda_P(D_1)$.

Finally, $L^{\prime}(N,P) = 0$ if and only if $N = \lambda_P(D_1)$.
By LaSalle's extension theorem \cite{LaSalleLefschetz1961}, any trajectory of system \eqref{NPZsys} in the
plane $Z=0$ for which $P_0>0$ approaches the largest invariant set in the line $N= \lambda_P(D_1)$, and
this is simply $\{E_1(\mu, D_1)\}$.
Therefore, $E_1(\mu, D_1)$ is globally asymptotically stable in the plane $Z=0$. 
\end{proof}
When $\mu = \mu_{c_1}(D_1, D_2)$, $E_1(\mu, D_1)$ and $E_2(\mu, D_1, D_2)$ coalesce.
With $\mu = \mu_{c_1}(D_1, D_2)$,  we have
$P\bigl(\mu_{c_1}(D_1, D_2), D_1\bigr) = \lambda_Z(D_2)$ (see equations \eqref{mu_{c1}} and \eqref{Poneformula}).
Also $N(\mu, D_1, D_2)=\lambda_P(D_1)$, so that $Z(\mu, D_1, D_2)=0$ (see the discussion around \eqref{valueZ_2}).
Thus,
\begin{equation*}
E_2(\mu_{c_1}(D_1, D_2), D_1, D_2) = (\lambda_P(D_1), \lambda_Z(D_2), 0) = E_1(\mu_{c_1}(D_1, D_2), D_1).
\end{equation*}
Said another way, as $\mu$ increases through $\mu_{c_1}(D_1,D_2)$,
$E_2(\mu, D_1, D_2)$ enters the positive cone by passing through $E_1(\mu, D_1)$.

With $f_2\bigl(\lambda_Z(D_2)\bigr) = D_2/\gamma_2$, the Jacobian matrix at $E_2(\mu, D_1, D_2)$  takes the form
\begin{multline} \label{JacobianE2}
J\bigl(E_{2}(\mu, D_1, D_2)\bigr) =
\\
 \begin{bmatrix}
  -D{-}\lambda_Z(D_2)f_1'(N(\mu,\! D_1,\! D_2)) & \mspace{-15mu} -f_{1}(N(\mu, \! D_1,\! D_2)) & \mspace{-1mu} 0 
\\
\gamma_{1}\lambda_{Z}(D_2)f_1'(N(\mu,\! D_1,\! D_2)) 
           &\mspace{-15mu}
           \begin{matrix}
             \Bigl( \gamma_{1}f_{1}(N(\mu,\! D_1, \!D_2)){-}D_{1}  \mspace{15mu}
\\
                   \mspace{45mu}    {-}Z(\mu,\! D_1,\! D_2)
                   f_2'(\lambda_{Z}(D_2)) \Bigr)
           \end{matrix}
                                                                          &\mspace{-1mu}  -D_2/\gamma_2
\\
0 & \mspace{-15mu} \gamma_{2}Z(\mu,\! D_1,\! D_2) f_2'(\lambda_{Z}(D_2)) & \mspace{-1mu} 0
\end{bmatrix}.  
\end{multline}
The eigenvalues of $J\bigl(E_{2}(\mu, D_1, D_2)\bigr)$ satisfy
\begin{equation*}
x^3 +a_{1}x^2 + a_{2}x +a_{3} = 0,
\end{equation*}
where 
\begin{align}  %\label{JE2coefficients} 
 a_{1}(\mu, D_1, D_2) &= Z(\mu, D_1, D_2) f_2'\bigl(\lambda_{Z}(D_2)\bigr)+\lambda_{Z}(D_2)f_1'\bigl(N(\mu, D_1, D_2)\bigr) \notag
\\  &\quad \; -\gamma_{1}f_{1}(N(\mu, D_1, D_2))+D_{1}+D,       \label{JE2coefficienta1}      
\\
 a_{2}(\mu, D_1, D_2) &= \lambda_Z(D_2) Z(\mu, D_1, D_2) f_1'\bigl(N(\mu, D_1, D_2)\bigr) f_2'(\lambda_{Z}(D_2))  \notag
\\ 
&\quad \;+ D_2Z(\mu, D_1, D_2) f_2'\bigl(\lambda_{Z}(D_2)\bigr)+ D Z(\mu, D_1, D_2) f_2'\bigl(\lambda_{Z}(D_2)\bigr)  \notag
\\
&\quad \quad + D_{1}\lambda_{Z}(D_2)f_1'\bigl(N(\mu, D_1, D_2)\bigr) - 
                   D \gamma_{1}f_{1}\bigl(N(\mu, D_1, D_2)\bigr) + D D_{1}, \label{JE2coefficienta2}
\intertext{and}
a_{3}(\mu, D_1, D_2) 
   &=D_{2}Z(\mu, D_1, D_2) f_2'\bigl(\lambda_{Z}(D_2)\bigr)\Bigl(D+\lambda_{Z}(D_2) f_1'\bigl(N(\mu, D_1, D_2)\bigr) \Bigr).\label{JE2coefficienta3}
\end{align}
\begin{theorem}\label{coexistence_equilibrium_1}
The coexistence equilibrium point
\begin{equation*}
  E_2(\mu, D_1, D_2) = E_2\bigl(N(\mu, D_1, D_2),\lambda_{Z}(D_2),Z(\mu, D_1, D_2)\bigr)  
\end{equation*}
is asymptotically stable if and only if $a_1 > 0$ and $ a_1 a_2 > a_3$.
\end{theorem}
\begin{proof}
Since $a_{3}$ is positive, this follows from the Routh-Hurwitz criterion. 
\end{proof}
We conclude this section with a significant strengthening of Lemma \ref{positivity_boundedness}.
We use the concept of uniform persistence introduced in  \cite{Butler1986}.
\begin{theorem}
 \label{uniform persistence}
 Assume that $\mu > \mu_{c_1}(D_1, D_2)$. 
Then system \eqref{NPZsys} is uniformly persistent with respect to all solutions satisfying $P_0 > 0$ and $Z_0 > 0$.
\end{theorem}
\begin{proof}
Recall from Lemma \ref{positivity_boundedness} that
all solutions of system \eqref{NPZsys} for which $P_0 > 0$ and $Z_0 > 0$ are positive and bounded. 

We first show that $\liminf_{t \rightarrow \infty}{N(t) > 0}$. If $\liminf_{t  \rightarrow \infty} N(t) = 0$ 
and $\limsup_{t \rightarrow \infty}N(t) = 0$, then $\lim_{t \rightarrow \infty} N(t) = 0$.  But 
this is impossible,
%Suppose $\lim_{t \rightarrow \infty}{N(t)} = 0$; then  $\lim_{t \rightarrow \infty}{N'(t)} = 0$. 
for then it follows from the $N$-equation that $N'(t) \rightarrow \mu D > 0$ as $t \rightarrow \infty$
and this, in turn, contradicts the fact that $N(t)$ is bounded.

Now, suppose  $\liminf_{t \rightarrow \infty}{N(t) = 0}$ while  $\limsup_{t \rightarrow \infty}{N(t) > 0}$. 
Then there exists a sequence  $\{ \tau_n  \}_{n=1}^\infty$ of local minima of $N(t)$ 
satisfying $\tau_n \rightarrow \infty$ as $n \rightarrow \infty$. Thus,
\begin{enumerate}
    \item $N'(\tau_n) = 0$, since $\tau_n$ is a local minimum, and 
    \item $N(\tau_n) \rightarrow 0$ as $n \rightarrow \infty$, since  $\liminf_{t \rightarrow \infty}{N(t) = 0}$.
\end{enumerate}
From the $N$-equation we have
\begin{equation*}
       N'(\tau_n)  = \mu D - \Bigl( N(\tau_n)D + P(\tau_n)f_1\bigl(N(\tau_n)\bigr) \Bigr), 
\end{equation*}
so that
\begin{equation*}
     0 =  \bigabs{N'(\tau_n)} \geq \bigabs{\mu D} - \Bigabs{ N(\tau_n)D + P(\tau_n)f_1\bigl(N(\tau_n)\bigr) }
\end{equation*}
Rearranging and using the facts that $N(\tau_n)\rightarrow0$ as $n\rightarrow\infty$, $f_1$ is continuous and $f_1(0)=0$, we get
\begin{equation*} 
   0 =  \lim_{\tau_n \rightarrow \infty}  \Bigabs{ N(\tau_n)D + P(\tau_n)f_1\bigl(N(\tau_n)\bigr) } \geq \mu D > 0,
\end{equation*}      
a contradiction. Hence, $\liminf_{t \rightarrow \infty}{N(t) > 0}$.

Choose $X(0) = \bigl(N_0, P_0,Z_0\bigr) \in \bR^3_+$. 
Then  $\omega \bigl(X(0) \bigr)$ is a nonempty, compact invariant set with respect to system \eqref{NPZsys}. 
We claim $E_0 = (\mu,0,0)$ and $E_1(\mu, D_1) = (\lambda_P(D_1),P(\mu, D_1),0)$ are not in $\omega(X(0))$. 
Suppose $E_0 =(\mu, 0,0) \in \omega \bigl(X(0) \bigr)$. Since $\mu > \mu_{c_1}(D_1,D_2)$, $E_0$ is an unstable hyperbolic equilibrium point. 
By theorem \ref{NPZ stability E_0},
$E_0$ is globally asymptotically stable with respect to solutions initiating in the plane $P=0$.
Since $X(0) \not \in m^+(E_0)$, $\{ E_0\} \neq \omega \bigl(X(0) \bigr)$. 
By the Butler-McGehee lemma (see lemma A1, \cite{ButlerMcGeheelemma}), 
there exists $Q \in \bigl( m^+(E_0) \setminus \{ E_0\}\bigr)  \cap  \omega \bigl(X(0) \bigr)$, 
so that $\cl \mathcal{O}(Q) \subset \omega(X(0))$. 
For such an initial condition, the governing system is $\displaystyle{N'(t) = D\bigl(\mu -N(t)\bigr)}$, 
and $\displaystyle{Z'(t) = -D_2 Z(t)}$.
But  then $\mathcal{O}(Q)$ becomes unbounded  as $t \rightarrow - \infty$. 
This is a contradiction to the compactness of $\omega \bigl(X(0) \bigr)$,
and so $E_0 \not \in \omega \bigl(X(0) \bigr)$.

Suppose $E_1(\mu, D_1) = (\lambda_P(D_1), P(\mu, D_1),0)$ is in $\omega(X(0))$. 
Since $\mu > \mu_{c_1}(D_1,D_2)$, $E_1(\mu, D_1)$ is an unstable hyperbolic equilibrium point. 
By theorem \ref{NPZ stability of E_1} 
$E_1(\mu, D_1)$ is globally asymptotically stable with respect to solutions  initiating in the plane $Z = 0$. 
Since $X(0) \not \in m^+\bigl(E_1(\mu, D_1)\bigr)$, $\{ E_1(\mu, D_1)\} \neq \omega \bigl(X(0) \bigr)$. 
By the Butler-McGehee lemma, there exists 
$\widehat{Q} \in \bigl( m^+\bigl(E_1(\mu, D_1)\bigr) \setminus \{ E_1(\mu, D_1)\}\bigr)  \cap  \omega \bigl(X(0) \bigr)$,
so that 
$\cl \mathcal{O}(\widehat{Q}) \subset \omega(X(0))$. 
If $\widehat{Q} \in m^-(E_0)$, then 
$E_0 \in \cl{\mathcal{O}(\widehat{Q})} \subset \omega(X(0))$, a contradiction. 
Thus $\widehat{Q}  \not \in m^-(E_0)$,  and this implies  $\mathcal {O}(\widehat{Q})$ is unbounded as $t \rightarrow -\infty$, 
contradicting the compactness of $\omega(X(0))$. Therefore, $E_1(\mu, D_1) \not \in \omega(X(0))$. 

Suppose the system \eqref{NPZsys} is not persistent. 
Then there exists $\widetilde{Q} \in \omega(X(0))$ such that $\widetilde{Q} \in m^+(E_0)$ or $\widetilde{Q} \in m^+\bigl(E_1(\mu, D_1)\bigr)$. 
Then $\cl {\mathcal{O}(\widetilde{Q})} \subset \omega(X(0))$,
 which implies either  $E_0 \in \omega(X(0))$ or $E_1(\mu, D_1)  \in \omega(X(0))$, neither of which can be true. 
Thus,  $\liminf_{t \rightarrow \infty}P(t) > 0$ and   $\liminf_{t \rightarrow \infty}Z(t) > 0$, 
and it follows from the main result of \cite{Butler1986}  that system \eqref{NPZsys} is uniformly persistent.
\end{proof}
\section{Stability of the coexistence equilibrium}
\label{Coexistence_equilibrium_stability}
In this section we study the  coexistence equilibrium point
\begin{equation*}
 E_2(\mu, D_1, D_2) = \bigl(N(\mu, D_1, D_2), \lambda_Z(D_2), Z(\mu, D_1, D_2)\bigr)
\end{equation*}
first when $D{=}D_1{=}D_2$ and $\mu$ varies and then after relaxing the assumption $D{=}D_1{=}D_2$.
In particular, we study the eigenvalues of the Jacobians at the coexistence equilibria as $\mu$ increases.
To prepare for the study of the evolution of the equilibrium $E_2(\mu, D_1, D_2)$, we observe
the following consequence of the implicit function theorem
\cite[p.122]{LangAnalysisII}. 
\begin{lemma}\label{NZsmoothness}
For $D_1> 0$, $D_2>0$, and $\mu>\mu_{c_1}(D_1,D_2)$ there is a local parameterization of the locus of coexistence equilibria 
\begin{equation*}
 E_2(\mu, D_1, D_2) = \bigl(N(\mu, D_1, D_2), \lambda_Z(D_2), Z(\mu, D_1, D_2)\bigr) 
\end{equation*}
defined on an interval containing $\mu$ and a disc around $(D_1, D_2)$ 
that is smooth to the smaller of the  degree of smoothness of $f_1(N)$ and the
degree of smoothness of $f_2(P)$.
\end{lemma}
\begin{proof}
Set $P=\lambda_Z(D_2)$ in the $N$- and $P$- equations of system \eqref{NPZsys}.
With $f_2(\lambda_Z(D_2)) = D_2/\gamma_2$ from \eqref{breakeven}, let  
\begin{equation*}
   G_1(N, Z, \mu, D_1, D_2) = D(\mu - N ) - f_1(N)\lambda_Z(D_2)
\end{equation*}
and 
\begin{equation*}
  G_2(N, Z, \mu, D_1, D_2)  = \gamma_1f_1(N)\lambda_Z(D_2) - D_1\lambda_Z(D_2) - (D_2/\gamma_2) Z,
\end{equation*}
and define $G \colon \bR^5 \rightarrow \bR^2 $ by $G(N, Z, \mu, D_1, D_2) =
\bigl(G_1(N,Z,\mu, D_1, D_2), G_2(N,Z,\mu, D_1, D_2)\bigr)$.
Then we want to parametrize the set $G^{-1}(0,0)$.  

The derivative of $G$ is represented by the matrix
\begin{equation*}
  DG = 
  \begin{bmatrix}
    \partial G_1/\partial N & \partial G_1/\partial Z & \partial G_1 /\partial \mu 
                                                & \partial G_1/\partial D_1 & \partial G_1/\partial D_2
\\
   \partial G_2/\partial N & \partial G_2/\partial Z & \partial  G_2 / \partial \mu 
                                                & \partial G_2/\partial D_1 & \partial G_2/\partial D_2
  \end{bmatrix}.
\end{equation*}
Observe that, for fixed $\mu_0>\mu_{c_1}(D_1,D_2)$,  the first two columns of $DG$ at the point 
\begin{equation*}
(N(\mu_0, D_1, D_2), Z(\mu_0,D_1, D_2), \mu_0, D_1, D_2)  
\end{equation*}
evaluate to
 \begin{equation*}
  \begin{bmatrix}
    -D - f'_1\bigl(N(\mu_0, D_1, D_2)\bigr)\lambda_Z(D_2)  & 0  
\\
    \gamma_1 f_1'\bigl(N(\mu_0, D_1, D_2)\bigr) \lambda_Z(D_2)  & -D_2/\gamma_2 
  \end{bmatrix}.
\end{equation*}
Since $f_1'(N) > 0$, the first two columns are linearly independent  and so the implicit function
theorem applies.  There exists a ball  $B_1$ around $(\mu_0, D_1, D_2)$ and a ball $B_2$ around 
$\bigl(N(\mu_0, D_1, D_2), Z(\mu_0, D_1, D_2)\bigr)$ such that
for each $(\mu, D_1, D_2)$ in $B_1$ there is a unique point $\bigl(N(\mu, D_1, D_2), Z(\mu, D_1, D_2)\bigr)$ in $B_2$ such that 
$G\bigl( N(\mu, D_1, D_2) , Z(\mu, D_1, D_2), \mu \bigr) =0$.  
Moreover, the functions $N(\mu,D_1, D_2)$ and $Z(\mu,D_1, D_2)$ have the same degree of
differentiability as does $G$, which is the minimum of the degrees of differentiability of $f_1(N)$ and $f_2(P)$.  
To explain how the differentiability of $f_2$ enters, 
note that the computation of $\partial N/\partial D_2$ and $\partial Z/\partial D_2$ via $\partial G/\partial D_2$ involves 
$\lambda_Z'(D_2) = \bigl( \gamma_2\cdot f_2'(\lambda_P(D_2))\bigr)^{-1}$ by \eqref{lambdaderivs}.
This all gives us a smooth parameterization of the equilibrium locus, as desired.
\end{proof}
\begin{remark}
  From the point of view of calculus,  each of the independent variables $\mu$, $D_1$, $D_2$ is on an equal footing with 
the others.  However, viewed through the lens of the structure of system \eqref{NPZsys}, we can describe the variable $\mu$ 
as a control parameter and the variables $D_1$ and $D_2$ as experimental parameters fixed at some earlier time. 
This distinction informs our analysis where we first consider the model when $D{=}D_1{=}D_2$ as $\mu$ varies,
and subsequently vary $D_1$ and $D_2$ in a neighborhood of $D$.
\end{remark}
\begin{theorem} \label{boundedness of Z_2}
Fix $D_1$ and $D_2$ and suppose $\mu > \mu_{c_1}(D_1,D_2)$, so that the coexistence equilibrium point 
$ E_2(\mu, D_1, D_2) = \bigl(N(\mu, D_1, D_2), \lambda_Z(D_2), Z(\mu, D_1, D_2)\bigr) $ exists. 
Then
 \begin{enumerate}
     \item the $\mu$-derivatives $N'(\mu, D_1, D_2)$ and $Z'(\mu, D_1, D_2)$ are positive.
     \item $\lim_{\mu \rightarrow \infty} N(\mu, D_1, D_2) = \infty$ and  $Z(\mu, D_1, D_2)$ is bounded.
 \end{enumerate}
\end{theorem}
\begin{proof}  Since $D_1$ and $D_2$ are fixed throughout this proof, we drop
  these symbols and write simply $N(\mu)$ and $Z(\mu)$.
To show $N'(\mu) > 0$,
recall  equation \eqref{1steqnwithZ_2notzero} from theorem \ref{existence equilibria NPZ}:
\begin{equation*}
    0 = \bigl(\mu - N(\mu) \bigr)D - \lambda_Z(D_2) \cdot f_1\bigl(N(\mu) \bigr).
\end{equation*}
 Differentiating \eqref{1steqnwithZ_2notzero} with respect to $\mu$ and rearranging,  we get
\begin{equation}
    N'(\mu) \cdot \Bigl( D + \lambda_Z(D_2) f_1'\bigl(N(\mu)\bigr) \Bigr) = D.
\end{equation}
Since $f_1'(N)$ is positive, it follows that  $N'(\mu)$ is positive.

To show  $Z'(\mu) > 0$, 
%substitute  $E_2= \bigl(N(\mu), \lambda_Z, Z(\mu)\bigr)$ into 
set  $f_2\bigl(\lambda_Z(D_2)\bigr) = D_2/\gamma_2$ in equation \eqref{2ndeqnwithZ_2notzero}, obtaining
\begin{equation*}
    0 = \gamma_1 \lambda_Z(D_2) f_1\bigl(N(\mu) \bigr) -D_1\lambda_Z(D_2) - Z(\mu) \cdot (D_2/\gamma_2).
\end{equation*}
Differentiating  with respect to $\mu$ and rearranging, we get
\begin{equation*}
   Z'(\mu) = (\gamma_1 \gamma_2/D_2) \cdot \lambda_Z(D_2)  \cdot f_1'\bigl(N(\mu)\bigr)\cdot N'(\mu).
\end{equation*}
Since $N'(\mu)$  and $f_1'(N)$ are positive, it follows that   $Z'(\mu)$ is positive. This proves part one.

To prove part two, note that equation  \eqref{1steqnwithZ_2notzero} implies
\begin{equation*}
 N(\mu)  = \mu - (\lambda_Z(D_2)/D) \cdot f_1\bigl(N(\mu) \bigr).
\end{equation*}
Since $f_1(N)$ is bounded, $\lim_{\mu \rightarrow \infty} N(\mu) = \infty$.
From \eqref{2ndeqnwithZ_2notzero} we have
\begin{equation*}
     Z(\mu) =(\gamma_2/D_2)\cdot \bigl( \gamma_1 \lambda_Z(D_2) f_1\bigl(N(\mu) \bigr) -D_1\lambda_Z(D_2) \bigr) .
\end{equation*}
Since $f_1(N)$ is bounded,  $Z(\mu)$ is bounded. 
\end{proof}

Before we turn to a study of the stability properties of the coexistence equilibrium,
 we include a partial paraphrase of the $C^L$ Hopf bifurcation theorem as stated in \cite[p.16]{Hassard81}.
Since our goal is to make an application of this result to the coexistence equilibrium $E_2$,
and because the verification of the hypotheses is lengthy, 
we refer to our paraphrase to keep track of progress.
\begin{theorem} \label{HassardHopf}
  Consider a system $dX/dt = F(X,\mu)$ with $X \in \bR^n$ and $\mu$ a real parameter. If,
  \begin{enumerate}
  \item  for $\mu$ in an open interval containing $\mu_c$ (characterized in 3 below), 
$F(0,\mu) =0$  and $0 \in \bR^n$ is an isolated equilibrium point of $dX/dt{=}F(X,\mu)$;
\item all partial derivatives of the components $F^{\ell}$ of the vector $F$ of orders $\leq L{+}2$, ($L \geq 2$) exist
and are continuous in $X$ and $\mu$ in a neighborhood of $(0,\mu_c)$ in $\bR^n{\times}\bR$;
\item the Jacobian $J(0, \mu)= D_XF(0, \mu)$ has a pair of complex eigenvalues 
$\alpha(\mu){\pm} i\, \omega(\mu)$, where $\alpha(\mu_c) = 0$ and $\alpha'(\mu_c) \neq 0$;
\item the remaining $n{-}2$ eigenvalues of $J(0,\mu_c)$ have strictly negative real parts,
  \end{enumerate}
then the system $dX/dt = F(X, \mu)$ has a family of periodic solutions. 
\end{theorem}
\begin{remark}
  For the purposes of the proof in \cite{Hassard81} the authors assume the critical value of
the bifurcation parameter is $\mu_c = 0$, which is a trivial alteration. 
We are only interested in the 
existence  of cycles, so we do not quote several additional conclusions offered in \cite{Hassard81}.
In our situation 
the equilibrium $E_2$ depends on the parameter $\mu$, an issue which we circumvent 
in section \ref{Cycles}  using the inverse function theorem. Hypothesis 2
is fulfilled by imposing more differentiability conditions on the functions $f_1(N)$ and $f_2(P)$
at an appropriate point in the exposition.  
Verification of hypotheses 3 and 4 in the statement of theorem \ref{HassardHopf} is the most
involved part of our process and occupies the rest of this section. 
\end{remark}
To begin the stability analysis, we conjugate the Jacobian $J(E_2)$ in \eqref{JacobianE2} by %
\begin{equation*}
    W = \begin{bmatrix}
    1  & \gamma_1^{-1} &  (\gamma_1 \gamma_2)^{-1}
    \\
    0 & 1 & 0
    \\
    0 & 0  &  1
    \end{bmatrix}.
\end{equation*}
Using $f_2\bigl(\lambda_Z(D_2)\bigr) = D_2/\gamma_2$ and writing 
$D_1 = D{+}\epsilon_1$, $D_2 = D{+}\epsilon_2$ yields the matrix
\begin{equation*}
    WJ(E_2)W^{-1} =
    \\  \\
    \begin{bmatrix}
    -D   
         &   -\gamma_1^{-1}\epsilon_1     
              & -(\gamma_1 \gamma_2)^{-1}\epsilon_2
              \\
    \gamma_1 \lambda_Z(D_2) f_1'\bigl(N(\mu, D_1, D_2)\bigr)
    & A
    & B
    \\
    0  & C & 0
    \end{bmatrix},
\end{equation*}
where
\begin{align*}  
  A &= \gamma_1 f_1 \bigl(N(\mu, D_1, D_2)\bigr) 
            -  \lambda_Z(D_2) f_1' \bigl(N(\mu, D_1, D_2)\bigr)
                  - D_1 
                  - Z(\mu, D_1, D_2) f_2'\bigl(\lambda_Z(D_2)\bigr), 
\\
\begin{split}  B &=-( \lambda_Z(D_2) /\gamma_2 )  f_1'\bigl(N(\mu, D_1, D_2)\bigr) -  f_2\bigl(\lambda_Z(D_2)\bigr) 
\\
              &= - \bigl( \lambda_Z(D_2) f_1'\bigl(N(\mu, D_1, D_2)\bigr) + D_2 \bigr)/\gamma_2<0,    \end{split}
\\
   C &=\gamma_2 Z(\mu, D_1, D_2)f'_2\bigl(\lambda_Z(D_2)\bigr)>0. 
 \end{align*}
 First we make the assumption that  $D{=}D_1{=}D_2$, so that $\epsilon_1{=}\epsilon_2{=}0$; 
that is, we assume the death rates of $P$ and $Z$ are negligible with respect
to washout rate $D$. The results of
\cite{Kuang1999} suggest this is a useful initial assumption.  
Since we regard $D$ as fixed for the discussion, we will abbreviate $N(\mu, D, D)$ by
$N(\mu)$  and $Z(\mu, D, D)$ by $Z(\mu)$. Similarly, we will abbreviate
$E_2(\mu, D, D)$ by $E_2(\mu)$.

We can explicitly compute the eigenvalues of $J\bigl(E_2(\mu)\bigr)$, since conjugation does not change them.
The characteristic polynomial of $J\bigl(E_2(\mu)\bigr)$ is
\begin{equation}   \label{cpJE2DD}
    p(x) = (-D - x) ( -BC -Ax + x^2)
\end{equation}
where
\begin{align}
  A &= \gamma_1 f_1 \bigl(N(\mu)\bigr) -  \lambda_Z(D) f_1' \bigl(N(\mu)\bigr)- D - Z(\mu) f_2'\bigl(\lambda_Z(D)\bigr), \label{Aspecial}
\\
  B &= - \bigl( \lambda_Z f_1'\bigl(N(\mu)\bigr) + D \bigr)/\gamma_2,   \label{Bspecial}
\intertext{and}
 C &= \gamma_2 Z(\mu)f'_2\bigl(\lambda_Z(D)\bigr).   \label{Cspecial}
\end{align}
The next result amplifies theorem \ref{coexistence_equilibrium_1} in the case $D{=}D_1{=}D_2$.
\begin{theorem} \label{E^*_2}
Assume $D{=}D_{1}{=}D_{2}$ and that $\mu > \mu_{c_1}(D,D)$, so the coexistence equilibrium 
$E_2(\mu) = (N(\mu), \lambda_Z(D), Z(\mu))$ exists. Then
\begin{enumerate}
\item  $E_2(\mu)$ is  locally asymptotically stable  if 
\begin{equation*}
    Z(\mu) \Bigl(\frac{D}{\gamma_{2}\lambda_{Z}(D)}-f_{2}'\bigl(\lambda_{Z}(D)\bigr)\Bigr) < \lambda_{Z}(D) f_{1}'\bigl(N(\mu)\bigr).
\end{equation*}
\item $E_2(\mu)$ is  unstable if 
\begin{equation*}
    Z(\mu) \Bigl(\frac{D}{\gamma_{2}\lambda_{Z}(D)}-f_{2}'\bigl(\lambda_{Z}(D)\bigr) \Bigr) > \lambda_{Z}(D) f'_1\bigl(N(\mu)\bigr).
\end{equation*}
\end{enumerate}
\end{theorem}
\begin{proof}
With $f_2\bigl(\lambda_Z(D)\bigr) = D/\gamma_2$ in \eqref{2ndeqnwithZ_2notzero} we have 
$\gamma_1f_1\bigl(N(\mu)\bigr)-D = \bigl(DZ(\mu)\bigr)/\bigl(\gamma_2\lambda_Z(D)\bigr)$. Then 
\begin{align} 
    A&=  \gamma_{1} f_{1} (N(\mu)) - D -Z(\mu)f_{2}'\bigl(\lambda_Z(D)\bigr) - \lambda_Z(D) f_{1}'(N(\mu)) \notag \\
&= Z(\mu) \Bigl(\frac{D}{\gamma_{2}\lambda_Z(D)}-f_{2}'\bigl(\lambda_Z(D)\bigr) \Bigr) -\lambda_Z(D) f_{1}'(N(\mu)).\label{A} 
\end{align}
The result now follows from the Routh-Hurwitz criterion, since $-BC>0$ is easily verified 
from formulas \eqref{Bspecial} and \eqref{Cspecial}. 
\end{proof}
From the factorization of the characteristic polynomial of $J\bigl(E_2(\mu)\bigr)$ given in \eqref{cpJE2DD} it
is immediate that the Jacobian at $E_2(\mu)$ has one negative eigenvalue.
Thus,  hypothesis 4 of
theorem \ref{HassardHopf} for $E_2(\mu)$ is satisfied in the case $D{=}D_1{=}D_2$.  Now we verify hypothesis 3 
for this situation. 
\begin{theorem} \label{stability_special}
Assume $D{=}D_{1}{=}D_{2}$ and let $\mu > \mu_{c_1}(D,D)$, so that the coexistence equilibrium 
$E_2(\mu)= \bigl(N(\mu), \lambda_Z, Z(\mu)\bigr)$ exists. 
  \begin{enumerate}
     \item If $f'_1(N)$ is continuous and $D/\bigl(\gamma_2 \lambda_Z(D)\bigr)>f_2'\bigl(\lambda_Z(D)\bigr)$, 
then there exists a value $\mu_{c_2} > \mu_{c_1}(D,D) $ for which $A(\mu_{c_2}) = 0$.
Consequently, when $\mu = \mu_{c_2}$, the Jacobian has a conjugate pair of imaginary eigenvalues.
     \item 
If, in addition,  $f_1$ is twice differentiable with respect to $N$ and $f_1^{(2)} \bigl(N(\mu_{c_2})\bigr)< 0$, 
then $A'(\mu_{c_2}) > 0$. 
Combining this with part 1, we have that hypothesis
3 of theorem \ref{HassardHopf} is satisfied for $E_2(\mu)$.
\item If $f_1^{(2)}(N) < 0$ for all $N$, then $\mu_{c_2}$ is unique.
\end{enumerate}
\end{theorem}
\begin{remark}
  If we assume $f^{(2)}_2(P)  < 0$, i.e., that $f_2$ is is concave down, then
the  slope of the secant that passes through the points $(0,0)$
and $\bigl(\lambda_Z(D), (D/\gamma_{2})\bigr)$ is greater than 
the slope of the tangent line to the graph of $f_2$ at $\lambda_{Z}(D)$; 
that is,  $D/\bigl(\gamma_{2}\lambda_Z(D)\bigr) > f_{2}'\bigl(\lambda_{Z}(D)\bigr)$. 
This will  be the case, for example, when $f_2(P)$ is Holling Type II. 
However,  the one-point condition  
$D/\bigl(\gamma_{2}\lambda_Z(D)\bigr)-f_{2}'\bigl(\lambda_Z(D)\bigr) > 0$ 
may hold even if $f_2$ is not concave down. 
We will give such an example in Section \ref{Examples}.
\end{remark}
\begin{proof}
Consider the expression 
\begin{equation*}
  A(\mu)  = Z(\mu) \Bigl(\frac{D}{\gamma_{2}\lambda_Z(D)}-f_{2}'\bigl(\lambda_Z(D)\bigr) \Bigr) -\lambda_Z(D) f_{1}'\bigl(N(\mu)\bigr)
\end{equation*}
given in \eqref{A}.
Since $f_1'(N)$ is a continuous function, $A$ is a continuous function of $\mu$ by Lemma~\ref{NZsmoothness}. 
To prove that $A(\mu)$ has a zero value for some $\mu>\mu_{c_1}(D,D)$, it is enough to prove that $A(\mu)$ passes from negative to positive.
For $\mu =\mu_{c_1}(D,D)$, $Z\bigl(\mu_{c_1}(D,D)\bigr) =0$, so 
\begin{equation*}
A\bigl(\mu_{c_1}(D,D)\bigr) = -\lambda_Z(D) f_{1}'\bigl( N\bigl(\mu_{c_1}(D, D)\bigr) \bigr) < 0.   
\end{equation*}

To find a value of $\mu>\mu_{c_1}(D, D)$ at which $A(\mu)$ is positive, we use Theorem \ref{boundedness of Z_2}.  
We have that  $Z(\mu)$ is increasing and bounded for $\mu > \mu_{c_1}(D, D)$. 
Let ${Z_{\infty} = \sup_{\mu \geq \mu_{c_1}(D, D)}Z(\mu)}$. 
Then there exists an $M_1$ such that for $\mu > M_1$, $Z(\mu) > Z_{\infty}/2$. 
Since $f_1$ is bounded and increasing, $\lim_{N \to +\infty} f'_1(N) = 0$. 
Then, for any $\epsilon>0$, there exists an $N_{\epsilon}>0$ such that $0<f_1'(N) < \epsilon$ for all $N>N_{\epsilon}$.
With $\epsilon = \bigl(Z_{\infty}/2\lambda_Z(D)\bigr)\cdot\Bigl( D/\bigl(\gamma_2 \lambda_Z(D)\bigr) - f'_2\bigl(\lambda_Z(D)\bigr)\Bigr)$, 
this implies that there exists an $N^*$ such that, if $N > N^*$, then 
\begin{equation*}
 0 < f'_1(N) < \bigl(Z_{\infty}/2\lambda_Z(D)\bigr) \cdot\Bigl( D/\bigl(\gamma_2 \lambda_Z(D)\bigr) - f'_2\bigl(\lambda_Z(D)\bigr)\Bigr).
\end{equation*}
In addition, $N(\mu)$ is increasing without bound by theorem \ref{boundedness of Z_2}, 
so there is an $M_2 > \mu_{c_1}(D, D)$ such that, if $\mu > M_2$, then $N(\mu) > N^*$. 
Choose $\mu^* > \max\{ M_1,M_2\}$. Then
\begin{multline*}
A(\mu^*) = Z(\mu^*)\cdot \Bigl(D/\bigl(\gamma_{2}\lambda_Z(D)\bigr)-f_{2}'\bigl(\lambda_Z(D)\bigr) \Bigr) - \lambda_Z(D) f_{1}'\bigl( N(\mu^*) \bigr) 
\\
>   (Z_{\infty}/2)\cdot \Bigl( D/\bigl(\gamma_2 \lambda_Z(D)\bigr) - f'_2\bigl(\lambda_Z(D)\bigr)\Bigr)
\\ 
   -   \lambda_Z(D) \cdot \bigl(Z_{\infty}/2\lambda_Z(D)\bigr) \cdot \Bigl( D/\bigl(\gamma_2 \lambda_Z(D)\bigr) - f'_2\bigl(\lambda_Z(D)\bigr)\Bigr)
 = 0.
\end{multline*}
Since $A\bigl(\mu_{c_1}(D, D)\bigr) < 0$ and  $A(\mu^*) > 0$,
there is a number $\mu_{c_2} > \mu_{c_1}(D, D)$ such that $A(\mu_{c_2}) = 0$. 
Note that when $\mu = \mu_{c_2}$ 
the discriminant of the quadratic factor of the characteristic polynomial in \eqref{cpJE2DD} is
\begin{equation} \label{discriminantspecial}
  A(\mu_{c_2})^2 + 4\, B(\mu_{c_2})\cdot C(\mu_{c_2})
              = -4\,\bigl(\lambda_Z(D) f_1'\bigl(N(\mu_{c_2})\bigr)+D \bigr)\cdot \bigl(f_2'\bigl(\lambda_Z(D)\bigr) Z(\mu_{c_2})\bigr) < 0,
\end{equation}
so its roots are indeed purely imaginary. 
This proves part one.

For part two, by continuity of the discriminant as a function of $\mu$, there is a neighborhood of $\mu_{c_2}$ on which the discriminant
is negative. 
Continuing, differentiate $A(\mu)$ with respect to $\mu$ to obtain
 \begin{equation} \label{A'} 
 A'(\mu) = Z'(\mu) \Bigl( D/\bigl(\gamma_2 \lambda_Z(D)\bigr) - f_2'\bigl(\lambda_Z(D)\bigr) \Bigr) 
                 -\lambda_Z(D) f_1^{(2)}\bigl( N(\mu)\bigr)\cdot N'(\mu). 
\end{equation}
By theorem \ref{boundedness of Z_2}, $N'(\mu)$ and $Z'(\mu)$ are positive.  
By the hypotheses of the present theorem,  $D/\bigl(\gamma_2 \lambda_Z(D)\bigr) -  f_2'\bigl(\lambda_Z(D)\bigr) > 0$ 
and $f_1^{(2)}(N(\mu_{c_2})) < 0$. Thus,  $A'(\mu_{c_2}) > 0$.
Combining parts one and two means that hypothesis 3 of Theorem \ref{HassardHopf} holds at $\mu_{c_2}$
for the case $D{=}D_1{=}D_2$.

For part three, if $f_1^{(2)}(N) < 0$ for all $N$, then $A'(\mu) > 0$ for $\mu > \mu_{c_1}(D, D)$.
\end{proof}

Let us now discuss weakening the condition $D{=}D_1{=}D_2$. 
Intuitively, for $(D_1, D_2)$ sufficiently close
to $(D,D)$ the eigenvalues of $J\bigl(E_2(\mu, D_1, D_2)\bigr)$ should exhibit behavior
similar to those of $J\bigl(E_2(\mu, D, D)\bigr)$.
In particular, the equilibrium $E_2(\mu, D_1, D_2)$ should exhibit a similar loss  of stability. 
To make these considerations precise, we first require lemma \ref{product_decomposition}.  
\begin{lemma}
  \label{product_decomposition}
Let $P_1(x)^- = \{ (\alpha - x) \mid \alpha < 0\}$ be the space of polynomials of degree 1 in $x$, with leading coefficient $-1$ and negative 
constant term,  
let $P_2(x)^- = \{ \beta - \gamma x + x^2 \mid \gamma^2 - 4\beta < 0\}$ be the space of monic quadratic polynomials in $x$ with real coefficients
and having a complex conjugate pair of roots,
and let $P_3(x)^- = \{ p_0 + p_1x + p_2x^2 - x^3 \}$ be the space of cubic polynomials in $x$ with leading coefficient~$-1$ and real coefficients.

Then the multiplication map $M \colon P_1^- \times P_2^- \rightarrow P_3^-$ is locally a diffeomorphism.  
\end{lemma}
\begin{proof}
  Identify $P_3^-$ with Euclidean space using $p_0$, $p_1$,  and $p_2$ as coordinates, identify $P_1^-$ with an open subset of 
$\bR$ using $\alpha$ as the coordinate, and identify $P_2^-$ with an open subset of the plane $\bR^2$ using $\beta$ and $\gamma$ 
as coordinates. Then the map $M$ has the expression 
\begin{equation*}
  M(\alpha, \beta, \gamma) = \bigl(p_0(\alpha, \beta, \gamma), p_1(\alpha, \beta, \gamma), p_2(\alpha, \beta, \gamma)\bigr)
                           = \bigl(\alpha \beta, (-\alpha \gamma  -\beta),( \alpha + \gamma)\bigr).
\end{equation*}
The derivative, or Jacobian, of $M$ at $(\alpha, \beta, \gamma)$ is
\begin{equation*}
  DM = 
  \begin{bmatrix}
    \beta & \alpha & 0 
\\
    -\gamma & -1 & -\alpha 
\\
   1 & 0 & 1
  \end{bmatrix},
\end{equation*}
which fails to be invertible if and only if 
\begin{equation*}
  \det(DM) =-( \beta - \alpha \gamma + \alpha^2) = 0.
\end{equation*}
Should this occur, then $\alpha = \bigl(\gamma \pm \sqrt{\gamma^2 - 4\beta}\bigr)/2$.  
But we assume $\alpha < 0$ is real and $\gamma^2 - 4\beta < 0$, so $\det(DM) = 0$ is
impossible.  The map $M$ is smooth, so, by the inverse function theorem \cite[p.125]{LangAnalysisII}, 
it is locally a diffeomorphism.
\end{proof}
To explain how lemma \ref{product_decomposition} comes into play, consider the map 
$\chi \colon M_{3,3}(\bR) \rightarrow P_3^-$  which takes as input a real-valued $3$ by $3$ matrix $R$
and produces its characteristic polynomial 
$\chi(R) = \det(R - xI)$.
The map $\chi$  has a coordinate expression by taking the coefficients in degrees $0$, $1$, and~$2$.
These coefficients are polynomials in the matrix entries, so $\chi$ is smooth. Now look at
\begin{equation*}
 M_{3,3}(\bR) \stackrel{\chi}{\longrightarrow}  P_3^-    \stackrel{M}{\longleftarrow}  P_1^- \times P_2^-.
\end{equation*}
Suppose we are in the situation of  theorem \ref{stability_special}, 
where it is easy to factor the characteristic polynomial of
$J\bigl(E_2(\mu, D, D)\bigr)$, and we have seen the Jacobian 
 has a negative real eigenvalue and a complex conjugate pair of eigenvalues
as $\mu$ varies near a potential bifurcation value $\mu_{c_2}$. 
The factorization explicitly inverts the polynomial multiplication $M$ at  particular points, and
lemma \ref{product_decomposition} enables us to extend the factorization, in principle, 
to the characteristic polynomial of $J\bigl(E_2(\mu, D_1, D_2)\bigr)$.

In particular, we can understand how 
the characteristic polynomial of $J\bigl(E_2(\mu, D_1, D_2)\bigr)$ behaves as $\mu$ varies
when $(D_1,D_2)$ is close to $(D,D)$ (in the Euclidean norm, for definiteness).
We remind the reader that we think of $\mu$ as a control parameter, adjustable by
the experimenter, and $D_1$ and $D_2$ as experimental parameters, set at the beginning
of an experiment.
To bring out this distinction, we will write the components of the formal factorization of 
the characteristic polynomial of $J\bigl( E_2(\mu, D_1, D_2)\bigr)$ as
$\alpha(D_1, D_2)(\mu)$, $\beta(D_1, D_2)(\mu)$, and $\gamma(D_1, D_2)(\mu)$.
\begin{lemma}
  \label{uniformapproximation}
Assume $f_1$ is three times continuously differentiable and $f_2$ is two times continuously differentiable. Then
there exists a $\mu$-interval $[\mu{-}\hat{\delta}, \mu{+}\hat{\delta}]$ on which the $\mu$-derivative $\gamma'(D_1,D_2)(\mu)$ is uniformly approximated by
$\gamma'(D,D)(\mu) = A'(\mu)$.  In fact, there exists a constant $C$ such that 
\begin{equation*}
   \abs{ \gamma'(D_1, D_2)( \mu) - \gamma'(D,D)(\mu)} \leq C \cdot \Dist{(D_1, D_2)}{(D, D)}
\end{equation*}
for any $\mu \in [\mu_{c_2}{-}\hat{\delta}, \mu_{c_2}{+}\hat{\delta}]$
\end{lemma}
The details of the proof of lemma \ref{uniformapproximation} are relegated to section \ref{Appendix} so as not
to disturb the flow of the exposition. 
\begin{theorem} \label{stability_general}
Let $\mu > \mu_{c_1}(D_1, D_2)$, so that the equilibrium point 
$E_2(\mu, D_1, D_2)$
exists in the interior of the positive octant.
Assume  $f_1$ is three times continuously differentiable and $f_2$ is two times continuously differentiable
and that $D/\bigl(\gamma_2 \lambda_Z(D)\bigr) > f_2'\bigl(\lambda_Z(D)\bigr)$
so that the hypotheses of theorem \ref{stability_special} part 1 are satisfied for $E_2(\mu, D, D)$,
and let $\mu_{c_2}$ be as in theorem \ref{stability_special}.

For $(D_{1}, D_{2})$ sufficiently close to $(D,D)$,
\begin{enumerate}
\item 
hypothesis 4 of theorem \ref{HassardHopf} holds for $E_2(\mu, D_1, D_2)$;
\item
and if, in addition,  $f_1^{(2)}\bigl(N(\mu_{c_2}, D, D)\bigr) <  0$ so that the hypotheses of
theorem \ref{stability_special} part 2 are satisfied, 
then hypothesis 3 of theorem \ref{HassardHopf} holds for $E_2(\mu, D_1, D_2)$.
\end{enumerate}
\end{theorem}
\begin{remark}
Theorem \ref{stability_special} gives a condition, namely, $f_1^{(2)} < 0$, on system \eqref{NPZsys} under which
there is a unique number  $\mu_{c_2}$ at which $J\bigl(E_2(\mu, D, D)\bigr)$ has a purely imaginary 
pair of eigenvalues meeting the transverality condition.  
However, it is not {\em a priori} clear that, in general,  there is precisely one number at which these 
properties of the eigenvalues hold.  
Therefore, for the theorem and its proof, choose one such number $\mu_{c_2}$ and fix it throughout the discussion.
\end{remark}
\begin{proof}[Proof of theorem \ref{stability_general}.]
  We have proved in theorem \ref{stability_special} that there is an interval of parameters $\mu$
in which the characteristic polynomials
  \begin{equation*}
     p(x) = (-D -x) (-BC -Ax + x^2)
  \end{equation*}
 of $J\bigl(E_2(\mu, D, D)\bigr)$ have a complex conjugate pair of roots in addition to the 
eigenvalue $-D{<}0$, so they  are in the image of
the multiplication map $M$.  For $(D_1, D_2)$ close to $(D,D)$ the entries in $J\bigl(E_2(\mu, D_1, D_2)\bigr)$
are close to the entries in $J\bigl(E_2(\mu, D, D)\bigr)$, so the characteristic polynomial of $J\bigl(E_2(\mu, D_1, D_2)\bigr)$
is close to the characteristic polynomial of $J(E_2)(\mu, D, D)$.  
To see this explicitly, refer to the formulas 
\eqref{JE2coefficienta1}, \eqref{JE2coefficienta2}, and \eqref{JE2coefficienta3}; 
to account for the normalization of the characteristic polynomial
to leading coefficient $-1$ multiply each expression by~$-1$. 
Therefore, in view of lemma \ref{product_decomposition}, the characteristic polynomial
of $J\bigl(E_2(\mu, D_1, D_2)\bigr)$ has a decomposition of the same form as that of the characteristic polynomial of 
$J\bigl(E_2(\mu, D, D)\bigr)$.  
Written formally, the decomposition is 
\begin{equation*}
  M^{-1}\bigl(\chi \bigl( J\bigl(E_2(\mu, D_1, D_2)\bigr) \bigr) \bigr) = \bigl( \alpha(D_1, D_2)(\mu) - x , 
                             \beta(D_1, D_2)(\mu) - \bigl(\gamma(D_1, D_2)(\mu)\bigr)\, x + x^2 \bigr).
\end{equation*}
The map 
$M\colon P_1^-(x){\times}P_2^-(x) \rightarrow P_3^-(x)$ 
is infinitely differentiable, so the local inverse $M^{-1}$ is also.  
In particular, for a fixed value $\mu$, 
the coefficients of the decomposition are smooth functions of $(D_1,D_2)$ defined in a neighborhood of $(D,D)$. 

The first consequence is that, by definition of $P_1^-$, the characteristic 
polynomial of $J\bigl(E_2(\mu, D_1, D_2)\bigr)$ has a linear factor $\alpha(D_1, D_2)(\mu) - x$ with $\alpha(D_1, D_2)(\mu) < 0$. 
This shows that hypothesis 4 of theorem \ref{HassardHopf} can be satisfied. 

Now we start the verification of hypothesis 3 of theorem \ref{HassardHopf} for $E_2(\mu, D_1, D_2)$.
The discriminant of the quadratic factor of the characteristic polynomial of $J\bigl(E_2(\mu, D_1, D_2)\bigr)$
is 
\begin{equation*}
  \bigl(\gamma(D_1, D_2)(\mu)\bigr)^2 - 4\, \beta(D_1, D_2)(\mu).
\end{equation*}
At $\mu_{c_2}$ and for $(D_1, D_2)$ sufficiently close to $(D, D)$, 
this is close to the expression 
\begin{equation*}
   \bigl(\gamma(D, D)(\mu_{c_2})\bigr)^2 - 4\,\beta(D, D)(\mu_{c_2}) =  A(\mu_{c_2})^2 + 4\, B(\mu_{c_2})C(\mu_{c_2})
\end{equation*}
for the
discriminant of the quadratic factor of the characteristic polynomial of $J\bigl(E_2(\mu, D,D)\bigr)$
given in \eqref{discriminantspecial}, which is negative. 
Therefore, at $\mu_{c_2}$, the discriminant $\bigl(\gamma(D_1, D_2)(\mu)\bigr)^2 - 4\,\beta(D_1, D_2)(\mu)$ is also negative.
By continuity of the discriminant as a function of $\mu$, there is an interval $[\mu_{c_2}-\delta_0, \mu_{c_2}+\delta_0]$ on which
it is negative.  Therefore,  the characteristic polynomial of 
$J\bigl(E_2(\mu, D_1, D_2)\bigr)$  has a complex conjugate pair of roots on this interval. 

Under the assumptions that $f_1^{(2)}$ is continuous and $f_1^{(2)}\bigl(N(\mu_{c_2}, D, D)\bigr)< 0$, 
$A'(\mu)$ as calculated in 
\eqref{A'} is continuous and $A'(\mu_{c_2}) > 0$.
So there is a $\delta_1 > 0$ such that 
$A'(\mu) > A'(\mu_{c_2})/2$ on the interval $[\mu_{c_2}-\delta_1, \mu_{c_2}+\delta_1]$.
Choose $\delta= \min\{ \delta_0, \delta_1\}$.  
Put $\eta_1 = -A(\mu_{c_2}{-}\delta)/2>0$.
By continuity of $\gamma$ as a function of $(D_1, D_2)$, 
there is a $\rho_1 > 0$ such that $\Dist{(D_1,D_2)}{(D, D)} < \rho_1$ implies
\begin{align*}
\abs{ \gamma(D_1, D_2)(\mu_{c_2}{-}\delta) - \gamma(D, D)(\mu_{c_2}{-}\delta) } &< \eta_1.
\intertext{Remembering that $\gamma(D, D)(\mu) = A(\mu)$, we have}
  \abs{ \gamma(D_1, D_2)(\mu_{c_2}{-}\delta) - A(\mu_{c_2}{-}\delta) } &< \eta_1,
\\
   \gamma(D_1,D_2)(\mu_{c_2}{-}\delta) - A(\mu_{c_2}{-}\delta) &< -A(\mu_{c_2}{-}\delta)/2,
\\
   \gamma(D_1, D_2)(\mu_{c_2}{-}\delta) &< A(\mu_{c_2}{-}\delta)/2 < 0.
\end{align*}
Similarly, put $\eta_2 = A(\mu_{c_2}{+}\delta)/2>0$.
There is a $\rho_2 > 0$ such that $\Dist{(D_1,D_2)}{(D, D)} < \rho_2$ implies
\begin{align*}
\abs{ \gamma(D_1, D_2)(\mu_{c_2}{+}\delta) - \gamma(D, D)(\mu_{c_2}{+}\delta) } &< \eta_2,
\\
  \abs{ \gamma(D_1, D_2)(\mu_{c_2}{+}\delta) - A(\mu_{c_2}{+}\delta) } &< \eta_2,
\\
   -A(\mu_{c_2}{+}\delta)/2  &< \gamma(D_1, D_2)(\mu_{c_2}{+}\delta) - A(\mu_{c_2}{+}\delta) 
\\
   0 < A(\mu_{c_2}{+}\delta)/2 &< \gamma(D_1, D_2)(\mu_{c_2}{+}\delta).
\end{align*}
Now using the continuity of $\gamma(D_1, D_2)(\mu)$ as function of $\mu$ to combine these results, 
we find $\gamma(D_1, D_2)(\mu)$ has a zero in the interval $[\mu_{c_2}-\delta, \mu_{c_2}+\delta]$,
provided that $\Dist{(D_1,D_2)}{(D, D)} < \min\{ \rho_1, \rho_2 \}$. 

Combining the results of the previous two paragraphs, on the interval
$[\mu_{c_2} - \delta, \mu_{c_2} + \delta]$ and for $(D_1,D_2)$ sufficiently close to $(D,D)$,
the characteristic polynomial of $J\bigl(E_2(\mu, D_1, D_2)\bigr)$ has a complex conjugate pair of roots 
and at least one pair of purely imaginary roots. Moreover, by choice of $\delta$, $\gamma'(D, D)(\mu) = A'(\mu) > A'(\mu_{c_2})/2>0$ .

To see that  the transversality condition holds, 
let $\eta_3 = A'(\mu_{c_2})/2 > 0$.  By lemma \ref{uniformapproximation}, if $(D_1,D_2)$ is sufficiently close
to $(D,D)$,  on the interval 
$[\mu_{c_2}- \delta, \mu_{c_2}+ \delta]$ we have
\begin{align*}
  \abs{ \gamma'(D_1, D_2)(\mu) - \gamma'(D, D)(\mu)} &\leq \eta_3,
\\
 - A'(\mu_{c_2})/2 &<  \gamma'(D_1, D_2)(\mu) - \gamma'(D, D)(\mu),
\\
  \gamma'(D, D)(\mu)  - A'(\mu_{c_2})/2 &<  \gamma'(D_1, D_2)(\mu),
\\
  0  &< \gamma'(D_1, D_2)(\mu),
\end{align*}
since $\gamma'(D, D)(\mu) = A'(\mu) > A'(\mu_{c_2})/2$ 
on the interval $[\mu_{c_2}- \delta, \mu_{c_2}+ \delta]$, in particular, 
at the point where $\gamma(D_1, D_2)(\mu)$ has a zero.
This completes the proof that 
hypothesis 3 of theorem \ref{HassardHopf} holds for $E_2(\mu, D_1, D_2)$.
\end{proof}
\section{Bifurcation to cycles}
\label{Cycles}
In \cite{Kuang1999} it is shown that cycles exist for certain values of
parameters and under the assumption that $D{=}D_1{=}D_2$.  In a nutshell, the
assumption  implies that the $\omega$-limit set of a solution starting near the unstable equilibrium $E_2(\mu, D, D)$
is contained in a plane in $NPZ$-space.  Of course, this plane also contains $E_2$. 
  The authors observe that there is a two-dimensional attracting set for this system. 
The Poincar\'{e}-Bendixson theorem applies to the two-dimensional limit system, delivering 
the existence of a cycle for the $(N,P,Z)$-system. 

In this section we fix $(D_1, D_2)$ sufficiently close to $(D, D)$ so that
theorem \ref{stability_general} applies, and we want to apply the Hopf bifurcation theorem
\cite{Hassard81, MarsdenMcCracken76}, restated in theorem \ref{HassardHopf},
to deduce that the system~\eqref{NPZsys} undergoes a Hopf bifurcation as
$\mu$ passes a critical value.
In section \ref{Coexistence_equilibrium_stability} we studied the characteristic polynomial
of $J\bigl(E_2(\mu, D_1, D_2)\bigr)$ relative to the characteristic polynomial of $J\bigl(E_2(\mu, D, D)\bigr)$.  
Features of these polynomials continue to play a role.
 The results of section~\ref{Coexistence_equilibrium_stability}
show that for parameter values $(D_1, D_2)$ near $(D,D)$, 
hypotheses 3 and 4 of theorem \ref{HassardHopf} concerning the behavior of the eigenvalues of the 
Jacobian $J\bigl(E_2(\mu, D_1, D_2)\bigr)$ are satisfied.   
However, in system \eqref{NPZsys}, the coordinates of the equilibrium 
$E_2\bigl( N(\mu, D_1, D_2), \lambda_Z(D_2), Z(\mu, D_1, D_2)\bigr)$ are changing
with the parameter $\mu$, so hypothesis 1 of theorem \ref{HassardHopf} is not satisfied.  
The immediate goal of this section is to overcome this difficulty by using the inverse function theorem.
\begin{figure}[h!]
  \centering
 \fbox{ \includegraphics[width=0.5\linewidth,viewport=20 360 475 695,clip=true]{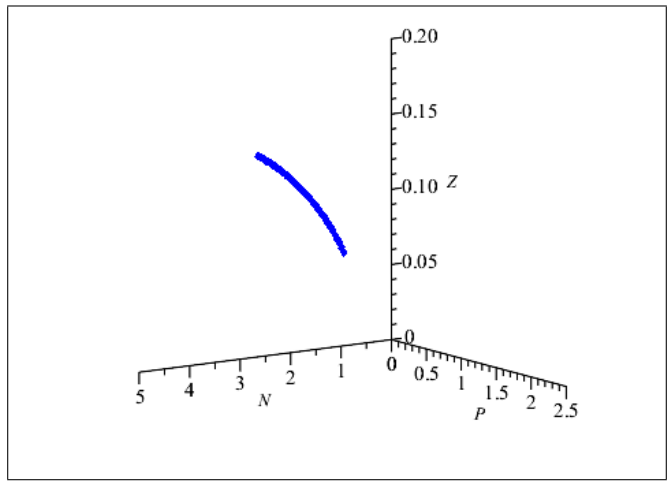} }
  \caption{A curve of coexistence equilibria}
  \label{coexisteqcurveh3}
\end{figure}
Augment  system \eqref{NPZsys} by introducing the bifurcation parameter as an extra
variable, giving 
\begin{equation} \label{aug_system} \begin{split}
  dN/dt &= D(\mu - N ) - f_1(N)P
\\
  dP/dt &= \gamma_1f_1(N)P - D_1P - f_2(P) Z
\\
  dZ/dt &= \gamma_2f_2(P)Z - D_2Z
\\
  d\mu/dt &= 0
\end{split}
\end{equation}
We are interested in the equilibria of system \eqref{aug_system} as $\mu$ varies in a small interval
around a number $\mu_{c_2}$ (depending on $D_1$ and $D_2$, but not necessarily uniquely determined)
 characterized in the proof of theorem \ref{stability_general} as a parameter value at which
the Jacobian $J\bigl(E(\mu_{c_2}, D_1, D_2)\bigr)$ has a  pair of purely imaginary eigenvalues,
crossing the imaginary axis in the complex plane transversally. 
By lemma \ref{NZsmoothness} the functions $N(\mu, D_1, D_2)$ and $Z(\mu, D_1, D_2)$ are smooth to
the same degree of smoothness possessed by $f_1(N)$ and $f_2(P)$. 

With $\delta$ as in the proof of theorem \ref{stability_general}, 
let $I$ be an interval containing $\mu_{c_2}$, contained in $[\mu_{c_2} - \delta, \mu_{c_2} + \delta]$, 
and supporting a curve
\begin{equation*}
h \colon I \rightarrow \bR^4,  \quad h(\mu) = \bigl(N(\mu, D_1, D_2), \lambda_Z(D_2), Z(\mu, D_1, D_2), \mu \bigr)
\end{equation*} 
parameterizing the equilibrium locus of the augmented system
\eqref{aug_system} near the critical point $h(\mu_{c_2})$.
We note that
\begin{equation*}
h'(\mu) = \bigl( N'(\mu, D_1, D_2), 0, Z'(\mu, D_1, D_2)\bigr), 1\bigr) \neq (0,0,0,0),
\end{equation*}
so $h$ is an immersion of an interval into $\bR^4$.

Consider next the map $H \colon \bR^3 \times I \rightarrow \bR^4$ defined by 
\begin{equation*}
  H(x,y,z, \mu) = (x,y,z, 0) + h(\mu)= (x+N(\mu, D_1, D_2), y+\lambda_Z(D_2), z + Z(\mu, D_1, D_2), \mu).
\end{equation*}
Observe that  $H(0,0,0, \mu) = \bigl(N(\mu,D_1,D_2),\lambda_Z(D_2), Z(\mu,D_1,D_2), \mu \bigr)$ and that
\begin{equation*}
  DH\bigl(0,0,0, \mu)\bigr) = 
  \begin{bmatrix}
    1 & 0 & 0 &   N'(\mu, D_1, D_2))
\\ 
    0 & 1 & 0 & 0
\\
    0 & 0 & 1 &  Z'(\mu, D_1, D_2))
\\
  0 & 0 & 0 & 1
  \end{bmatrix},
\end{equation*}
which is an invertible matrix for any $\mu$. 
In particular, on a neighborhood $U$ of $(0,0,0,\mu)$, $H$ is a diffeomorphism
of $U$ onto its image, smooth to the degree of smoothness of $N$ and $Z$, by the 
inverse function theorem \cite[p.122]{LangAnalysisII}.    
Consequently, an interval of the form $(0,0,0){\times}I'$
is mapped smoothly and bijectively onto the equilibrium locus of the system
\eqref{aug_system}.

For $Y = (N, P, Z, \mu)$ in $H(U)$ we can write $Y = H(X)$, where $X = (x, y, z, \mu)$.  
To simplify the notation, write  $dY/dt = F(Y)$ for system \eqref{aug_system}.
Then
  \begin{align}
    \frac{dY}{dt}& = DH(X)\cdot \frac{dX}{dt} = F(Y) = (F\circ H)(X), \notag
\intertext{so}
\frac{dX}{dt} &= DH(X)^{-1}\cdot (F\circ H)(X) \label{aug_system_abbr}
  \end{align}
is a formal expression for the system with respect to the new coordinates.  
If $Y_0$ is an equilibrium solution of system \eqref{aug_system},
and $H(X_0) = Y_0$, then $X_0$ is an equilibrium solution of the system \eqref{aug_system_abbr}.  

Let us now examine the relation between $DF(Y_0)$, the Jacobian of system \eqref{aug_system} at $Y_0$,
 and the Jacobian of the system \eqref{aug_system_abbr} at $X_0$. We compute
 \begin{multline*}
       D\bigl( DH(X)^{-1} \cdot F \circ H(X) \bigr)\bigr\rvert_{X = X_0} 
\\=
        D\bigl(X \mapsto DH(X)^{-1}\bigr)\bigr\rvert_{X = X_0} \cdot F \circ H (X_0)
    +  DH(X_0)^{-1}\cdot D\bigl(X \mapsto F \circ H (X)\bigr) \bigr\rvert_{X = X_0},
\\ \shoveleft{
\text{applying the Leibniz rule,}}
\\
=  D\bigl(X \mapsto DH(X)^{-1}\bigr)\bigr\rvert_{X = X_0} \cdot 0
    +  DH(X_0)^{-1}\cdot DF\bigl( H (X_0)\bigr) \cdot DH(X) \bigr\rvert_{X = X_0},
\\ \shoveleft{\text{by the chain rule,}}
 \\= DH(X_0)^{-1}\cdot DF\bigl( H (X_0)\bigr) \cdot DH(X_0) = DH(X_0)^{-1}\cdot DF(Y_0 ) \cdot DH(X_0) .  
   \end{multline*}
Thus, the Jacobian of   system \eqref{aug_system_abbr} at the equilibrium $X_0$ is
simply a conjugate of the Jacobian of   system \eqref{aug_system} at the equilibrium $Y_0 = H(X_0)$. 
In particular, the eigenvalues of the Jacobians are the same. We can now prove our main theorem.
\begin{theorem}
  \label{Hopfbifurcation}
Assume the hypotheses of theorem \ref{stability_general} hold and
that $f_1$ and $f_2$ are four times continuously differentiable.
For $D_1$ and $D_2$ both sufficiently close to $D$ there is a value
$\mu_{c_2}$ of the growth parameter at which the system \eqref{NPZsys}
undergoes a Hopf bifurcation, resulting in the appearance of a cycle. 
\end{theorem}
\begin{proof}
Let us now make the  assumption that $\Dist{(D_1, D_2)}{(D,D)}$ is so small that
the conclusions of theorem \ref{stability_general} hold, 
giving us a number $\mu_{c_2}$ (depending on $D_1$ and $D_2$, but not necessarily uniquely determined)
at which the Jacobian $J\bigl(E(\mu_{c_2}, D_1, D_2)\bigr)$ has a purely imaginary pair of eigenvalues,
crossing the imaginary axis in the complex plane transversally. 
The assumption that  $f_1$ and $f_2$ are four times continuously differentiable fulfills part two of theorem \ref{HassardHopf}. 
Now review elements of the proof pertaining to the eigenvalues
of the Jacobian of system \eqref{aug_system} at 
$(N(\mu, D_1, D_2), \lambda_Z(D_2), Z(\mu, D_1, D_2), \mu)$ as $\mu$ ranges over a small interval 
and increases through $\mu_{c_2}$.  Throughout, the Jacobian has a negative real eigenvalue, by 
theorem \ref{stability_general}, part one.
By part two of theorem  \ref{stability_general}, for $\mu < \mu_{c_2}$ and sufficiently near $\mu_{c_2}$, 
there is a complex-conjugate pair of eigenvalues with negative real part.
At $\mu = \mu_{c_2}$, the real part vanishes, and, for $\mu > \mu_{c_2}$  and sufficiently near $\mu_{c_2}$,  
there is a complex-conjugate pair of eigenvalues with positive real part. Moreover, the derivative
of the function selecting the real part of the complex-conjugate pair is positive 
at $\mu_{c_2}$. That is, the locus of the complex-conjugate pair of eigenvalues crosses the imaginary axis transversally.

By our observations on the relationship of the system \eqref{aug_system} 
to the system \eqref{aug_system_abbr},
as the parameter $\mu$ varies, 
the evolution of the eigenvalues of equilibria of system \eqref{aug_system_abbr} has the same characteristics.
Thus, hypotheses 4 and 3 of theorem \ref{HassardHopf} are satisfied for system \eqref{aug_system_abbr}.
Having assumed that $f_1$ and $f_2$ 
are four times continuously differentiable, then hypothesis 2  of theorem \ref{HassardHopf} is also satisfied.
Finally, since $(0,0,0)$ is an isolated equilibrium for all relevant values of $\mu$ for system \eqref{aug_system_abbr},
 hypothesis 1  of theorem \ref{HassardHopf} is satisfied.
The consequence is that the equilibrium $(0,0,0)$ undergoes a Hopf bifurcation at 
$\mu = \mu_{c_2}$, 
after which cycles appear in the phase portrait of system \eqref{aug_system_abbr}.  
Then the local diffeomorphism $H$ carries this portrait forward into the phase portrait of  system \eqref{aug_system}.
Thus, we have proved that cycles appear in our original system \eqref{NPZsys} when the parameter $\mu$ 
slightly exceeds $\mu_{c_2}$. 
\end{proof}

One can make similar observations viewing the Routh-Hurwitz expressions 
$a_1$, $a_2$, and $a_3$ as functions of $(D_1, D_2)$, but
it seems that this only delivers ``loss of stability,'' which is not quite enough.   
One can also play around with derivatives
of $J\bigl(E_2(\mu, D_1, D_2)\bigr)$ with respect to $D_1$ and $D_2$.   
One might get information about how far 
$(D_1, D_2)$ can vary from $(D,D)$ and still have a result.
\section{Simulations and Examples}
\label{Examples}
We consider now system \eqref{NPZsys} and take $f_1$ and $f_2$ to be
Michaelis-Menten functions, also called Holling type II functions.   We choose notations as follows.
\begin{equation}
    f_1(N) = m_1 N/(\alpha_1 + N) \quad \text{and} \quad f_2(P) = m_2P/(\alpha_2 + P). \label{HIIdefs}
  \end{equation}
With these choices explicit formulas can be given for many quantities studied in earlier sections. 
For example, we  obtain formulas for the numbers $\lambda_P(D_1)$ and $\lambda_Z(D_2)$ defined in equation \eqref{breakeven}.
For $\lambda_P(D_1)$ we have
\begin{equation}\label{HIIlambdaP}
  \gamma_1 \cdot \bigl(m_1 N/(\alpha_1 + N)\bigr) = D_1 
\quad \text{with solution} \quad 
N = D_1 \alpha_1/(\gamma_1 m_1 - D_1) \mathrel{\mathop :}= \lambda_P(D_1);
\end{equation}
for $\lambda_Z(D_2)$ we have
\begin{equation} \label{HIIlambdaZ}
   \gamma_2 \bigl( m_2 P/(\alpha_2 + P)\bigr) = D_2 
\quad \text{with solution} \quad 
P = D_2 \alpha_2/(\gamma_2 m_2 - D_2) \mathrel{\mathop :}= \lambda_Z(D_2).
\end{equation}
For rate functions of this type, the equation \eqref{1steqnwithZ_2notzero}  determining $N(\mu, D_1,D_2)$ reduces to 
a quadratic equation. In principle, one obtains explicit solutions for $N(\mu, D_1, D_2)$ and the nonnegative solution is easily identified.  
Then the equation \eqref{2ndeqnwithZ_2notzero} for $Z(\mu, D_1, D_2)$ is linear and easily solved for $Z(\mu, D_1, D_2)$.

With these remarks we can turn to an illustration of theorems \ref{stability_special} and \ref{stability_general} in a case
using Holling type II rate functions. 
With these rate functions, the concavity of the function $f_2$ guarantees that the condition 
$D/\bigl(\gamma_2 \lambda_Z(D)\bigr) > f_2'\bigl(\lambda_Z(D)\bigr)$ is satisfied.
First, we set $D{=}D_1{=}D_2{=}1.0 $ and the remaining parameters to the following values.
\begin{align*}
  m_1 &= 1 &  \alpha_1 &= 0.2 &  \gamma_1 &= 2 
\\
  m_2 &= 2 &   \alpha_2 &= 0.5 & \gamma_2 &=1.5
\end{align*}
With all quantities involved explicitly computed,  the formula given in \eqref{A} for the  real part
of the complex conjugate pair of eigenvalues of the linearization of the system at the coexistence equilibrium
can be made explicit, though messy,  and is easily plotted by a computer algebra system. This is the solid curve 
in figure \ref{realpartcompare}, which shows that a Hopf bifurcation occurs in the vicinity of $\mu = 0.6$. 
Figure \ref{h2Dsequal} exhibits solutions of the system for $\mu$ slightly smaller and slightly larger than the bifurcation
value $\mu_{c_2}(D, D) \approx 0.6$.

Next, keep $D{=}1$ and set $D_1{=}1.2$ and $D_2{=}1.3$.  Interpreted graphically, theorem \ref{stability_general} says
that the graph of the function defined by taking the real part of the complex conjugate pair of eigenvalues
associated with the coexistence equilibrium is an increasing function whose graph lies in a neighborhood 
of the curve we discussed in the preceding paragraph.
We do not have an explicit formula for {\em this} function with the new values for $D_1$ and $D_2$, but we can 
can estimate its values by numerically computing the eigenvalues of the linearization along a sequence of $\mu$-values.
Figure \ref{realpartcompare} also shows a sequence of points derived from eigenvalue approximations when $D_1{=}1.2$ and $D_2{=}1.3$. 
Interpolating a curve through the plotted points, we see a Hopf bifurcation occurs in the vicinity of $\mu = 0.9$.
Figure \ref{h2Dsunequal} shows trajectories of this system for values of $\mu$ slightly smaller and slightly larger
than the bifurcation value.

To provide an additional illustration of the result of theorem \ref{stability_general}, we consider a version
of system \eqref{NPZsys} incorporating rate functions with the property that the graphs have inflection points.
Consider
\begin{equation}
    f_1(N) = m_1 N^2/(\alpha_1 + N^2) \quad \text{and} \quad f_2(P) = m_2P^2/(\alpha_2 + P^2). \label{HIIIdefs}
  \end{equation}
with the parameter values
\begin{equation*}
  m_1 = 1.7, \quad \alpha_1 = 0.8, \quad  m_2 = 1.6  \quad  \alpha_2 = 0.9.
\end{equation*}
Further, set
\begin{equation*}
  \gamma_1 = 0.8,  \quad \gamma_2 = 0.9, \quad
D=1, \quad  D_1=1.2, \quad  D_2=1.1.
\end{equation*}
 Then the condition
$D/\bigl(\gamma_2 \lambda_Z(D)\bigr) > f_2'\bigl(\lambda_Z(D)\bigr)$
of theorem \ref{stability_general}
is satisfied, but this is not a consequence of the concavity of the graph of $f_2$. 

Determining $\lambda_P(D_1)$ and $\lambda_Z(D_2)$ in this example requires solving quadratic equations, so
obtaining exact values is quite easy.  Consequently, one can compute explicitly from \eqref{mu_{c1}} the value
$\mu_{c_1}(D_1, D_2)$  beyond which the coexistence equilibrium exists. Locating a coexistence equilibrium
$E_2(\mu, D_1, D_2)$ requires solving a cubic equation for $N(\mu, D_1, D_2)$, for which it is more appropriate
to use numerical methods.  
We choose a sequence of $\mu$ values starting beyond $\mu_{c_1}(D_1, D_2)$, 
approximate the coexistence equilibria and their Jacobian matrices $J\bigl(E_2(\mu, D_1, D_2)\bigr)$,
and numerically compute the real part of the complex pair of eigenvalues of the Jacobian for each $\mu$ value. 
 Plotting the real
part against $\mu$ produces figure~\ref{fig:hh3realpart}, which exhibits the expected change of sign
and shows that a Hopf bifurcation occurs in the vicinity of $\mu = 7.25$; trajectories 
for parameters slightly smaller and slightly larger than the bifurcation value are shown in figure~\ref{h3Dsunequal}.
\begin{figure}[h!]
  \centering
 \includegraphics[width=0.5\linewidth,viewport=20 360 475 695,clip=true]{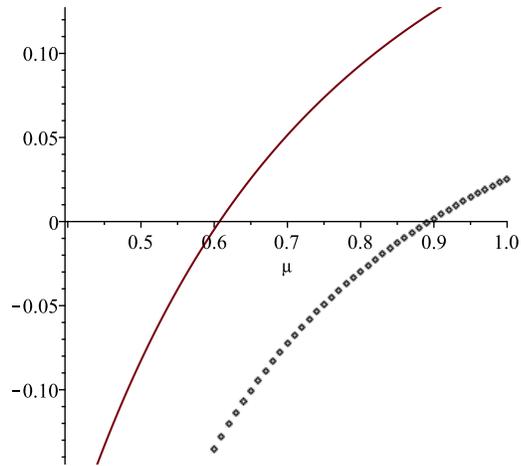}
  \caption{Comparison of real parts.  For the solid curve $D{=}D_1{=}D_2=1$; for the symbols $D{=}1$, $D_1{=}1.2$ and $D_2{=}1.3$.}
  \label{realpartcompare}
\end{figure}
\newpage
\begin{figure}[h!]
  \centering
  \begin{minipage}[c]{0.45\linewidth}
\centering
    \fbox{\includegraphics[width=0.80\linewidth,viewport=20 360 475 695,clip=true]{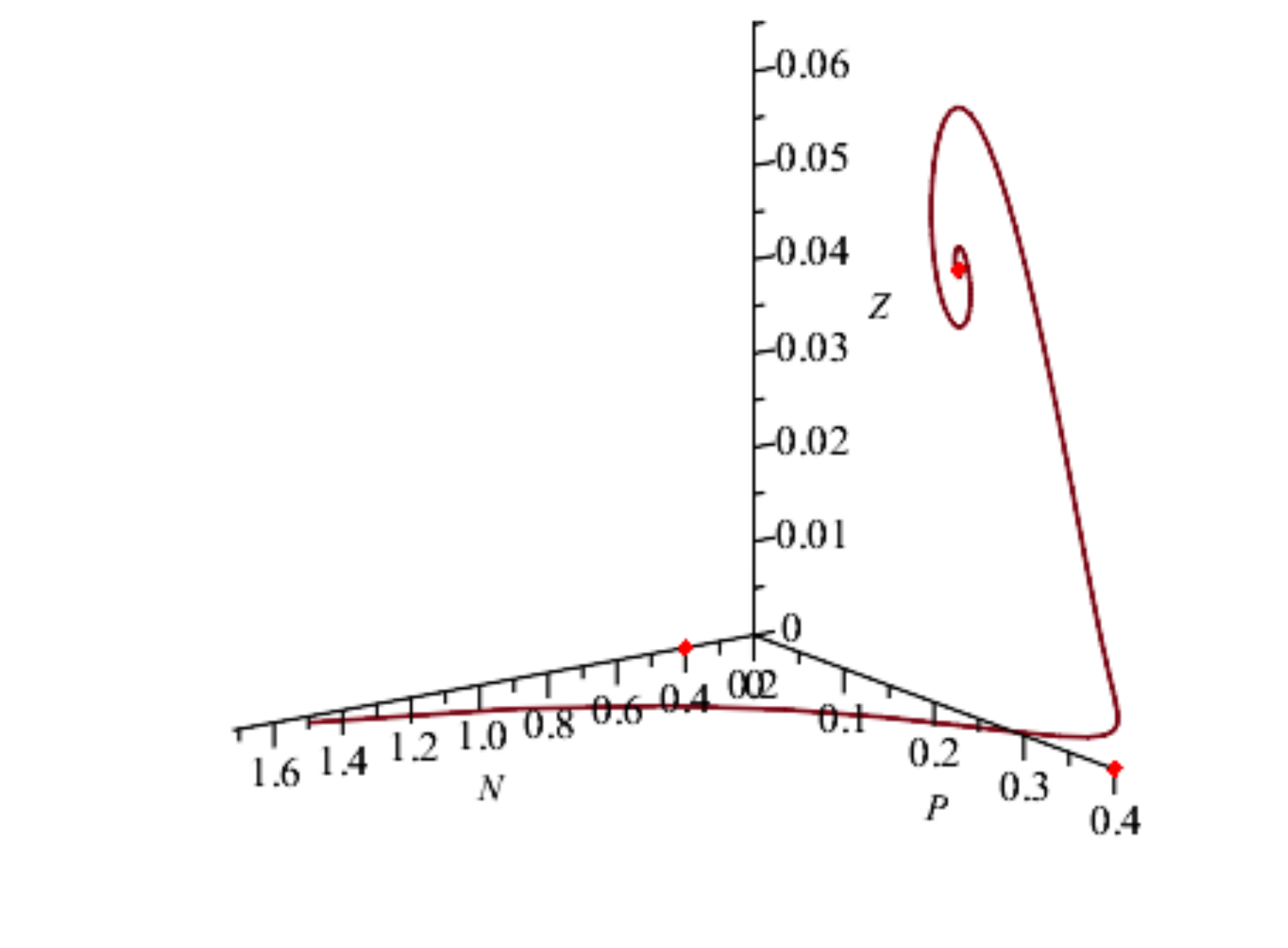}}
  \end{minipage}%
  \begin{minipage}[c]{0.45\linewidth}
\centering
     \fbox{\includegraphics[width=0.80\linewidth,viewport=20 360 475 695,clip=true]{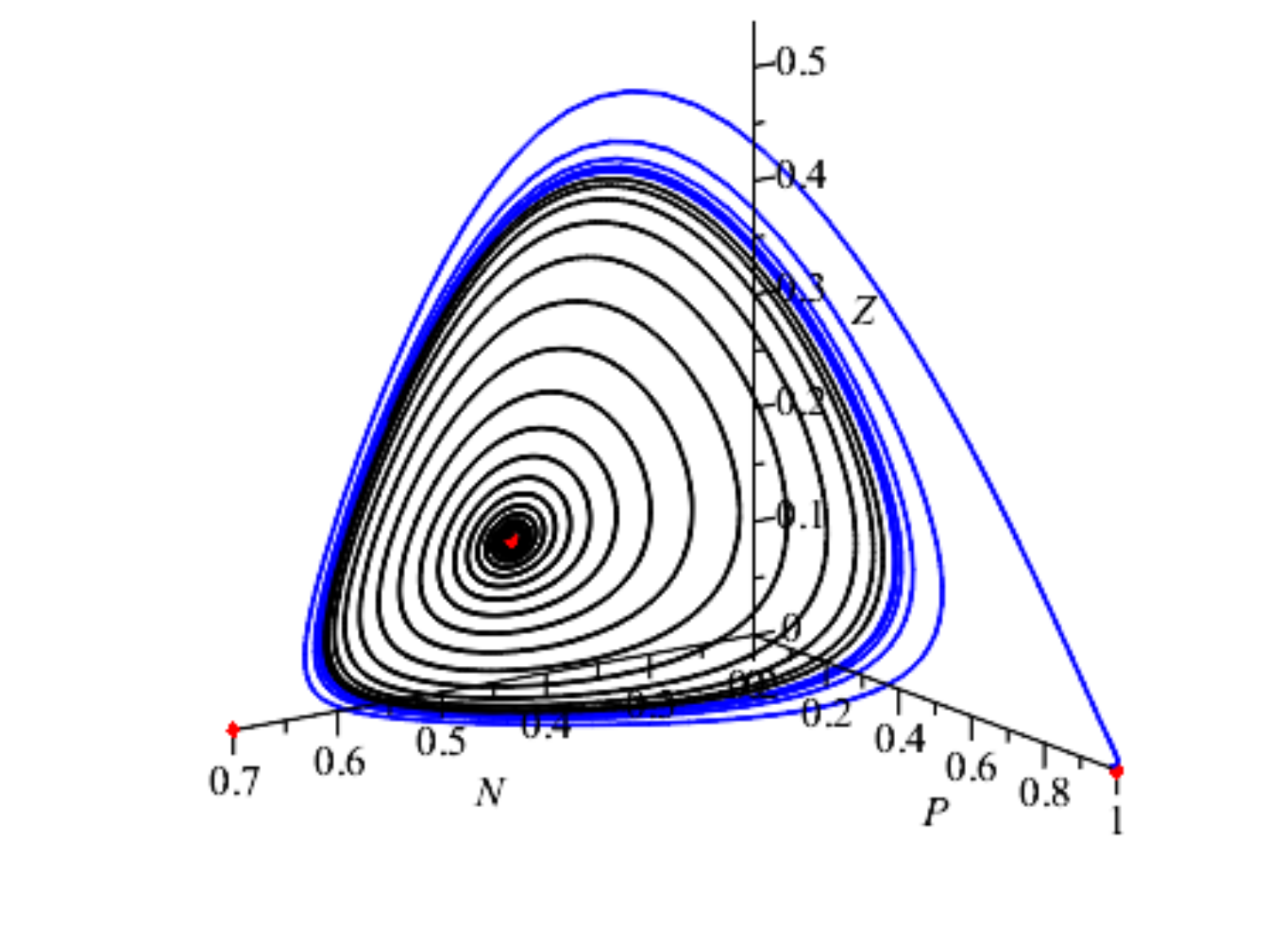}}
  \end{minipage}
\caption{Before and after a Hopf bifurcation: $D{=}D_1{=}D_2{=}1$ and Holling type II rate functions} \label{h2Dsequal}
\end{figure}
\begin{figure}[h!]
  \centering
  \begin{minipage}[c]{0.45\linewidth}
\centering
    \fbox{\includegraphics[width=0.80\linewidth,viewport=20 360 475 695,clip=true]{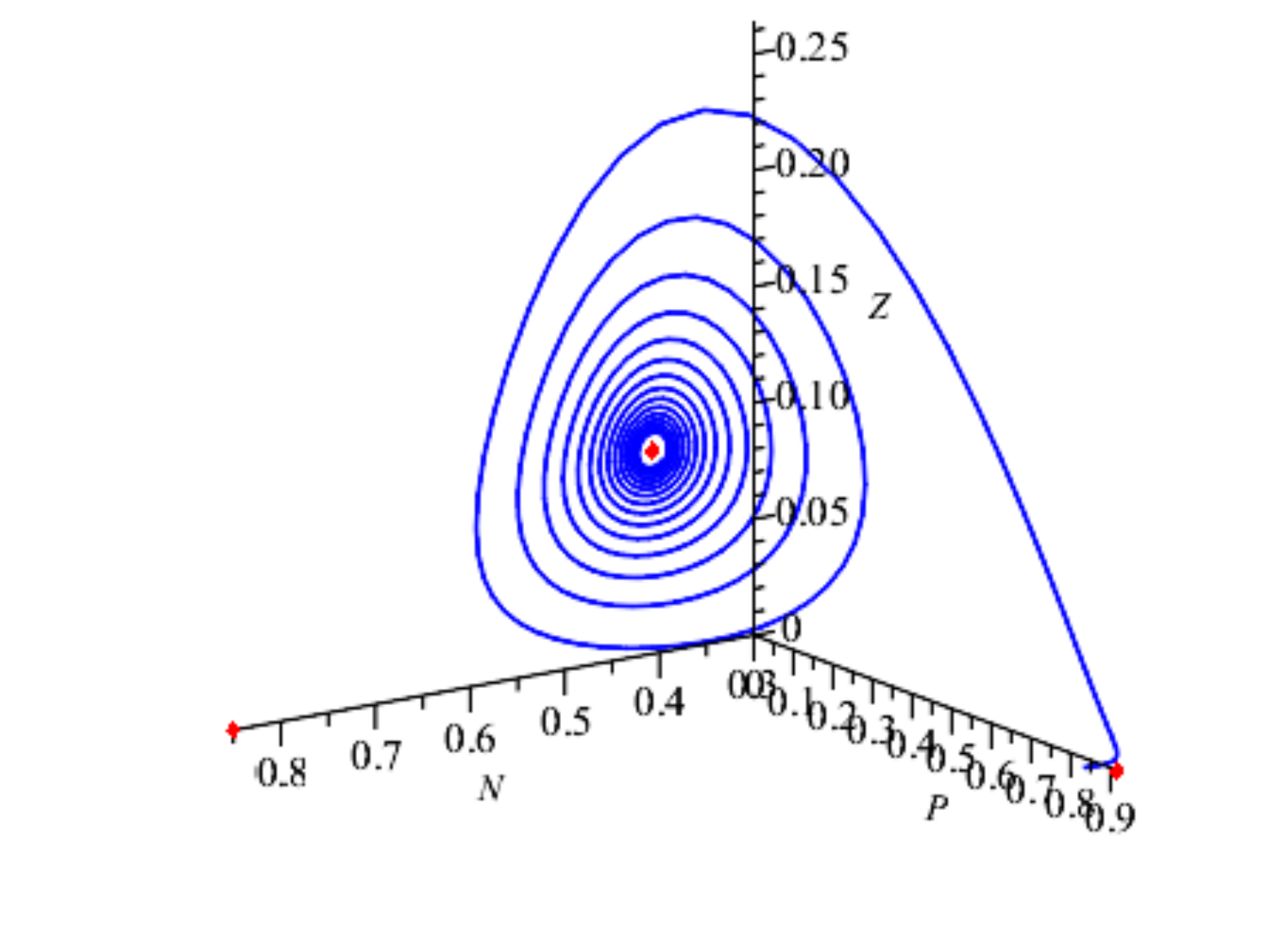}}
  \end{minipage}%
  \begin{minipage}[c]{0.45\linewidth}
\centering
     \fbox{\includegraphics[width=0.80\linewidth,viewport=20 360 475 695,clip=true]{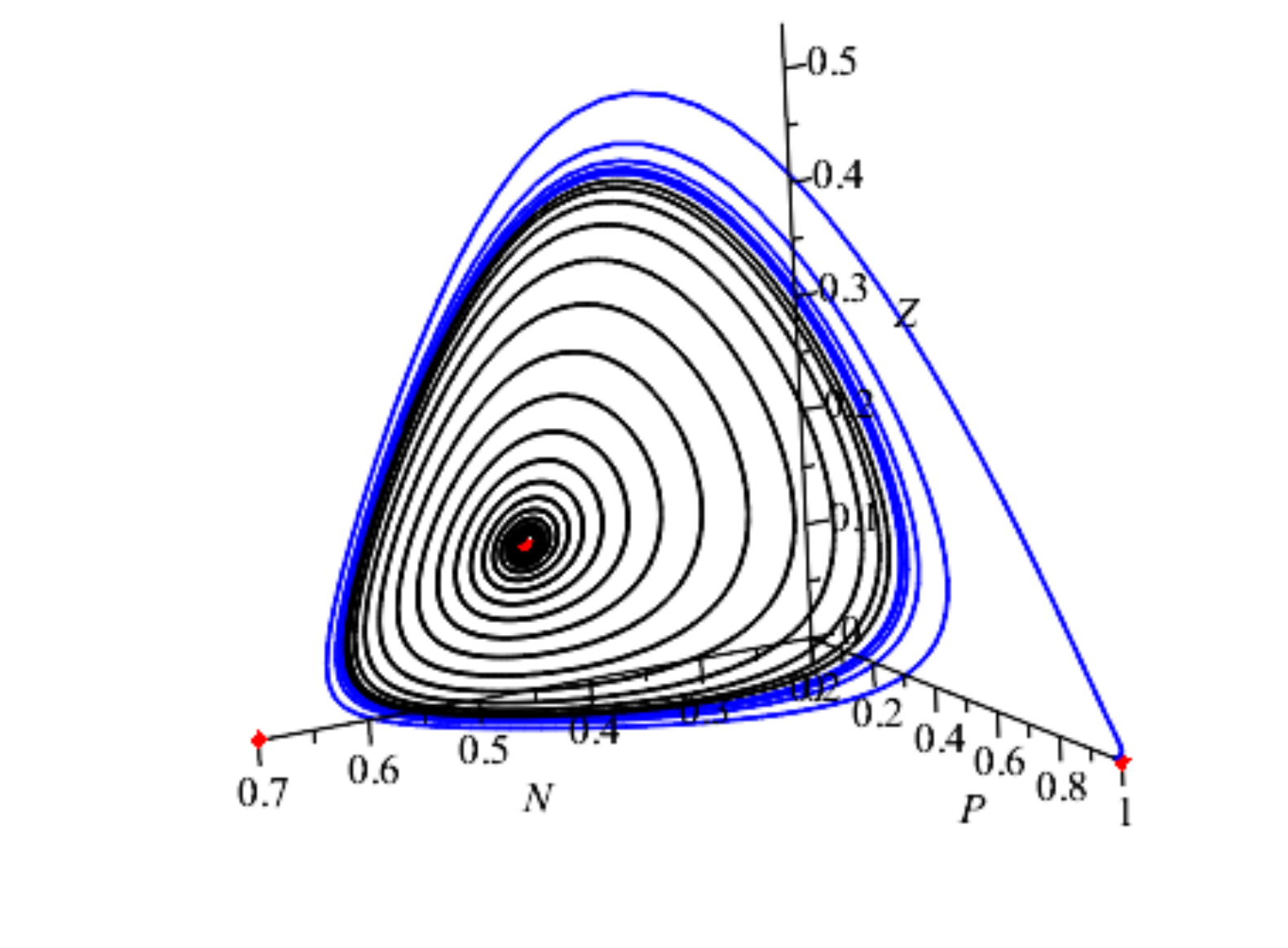}}
  \end{minipage}
\caption{Before and after Hopf bifurcation:  $D{=}1$, $D_1{=}1.2$ and $D_2{=}1.3$ and Holling type II rate functions} \label{h2Dsunequal}
\end{figure}
\begin{figure}[h!]
  \centering
  \includegraphics[width=0.5\linewidth,viewport=20 360 475 695,clip=true]{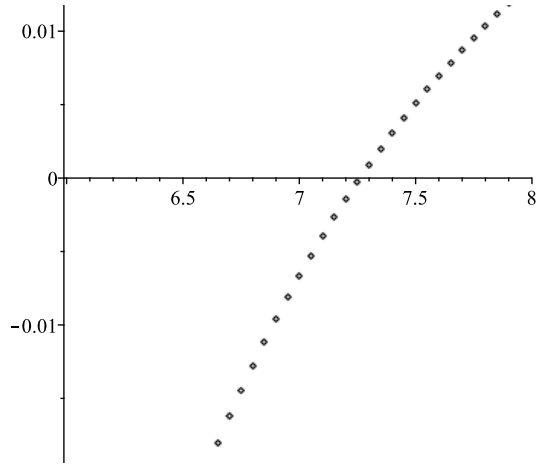}
  \caption{Real parts using rate functions \eqref{HIIIdefs}}
  \label{fig:hh3realpart}
\end{figure}
\begin{figure}[h!]
  \centering
  \begin{minipage}[c]{0.45\linewidth}
\centering
    \fbox{\includegraphics[width=0.80\linewidth,viewport=20 360 475 695,clip=true]{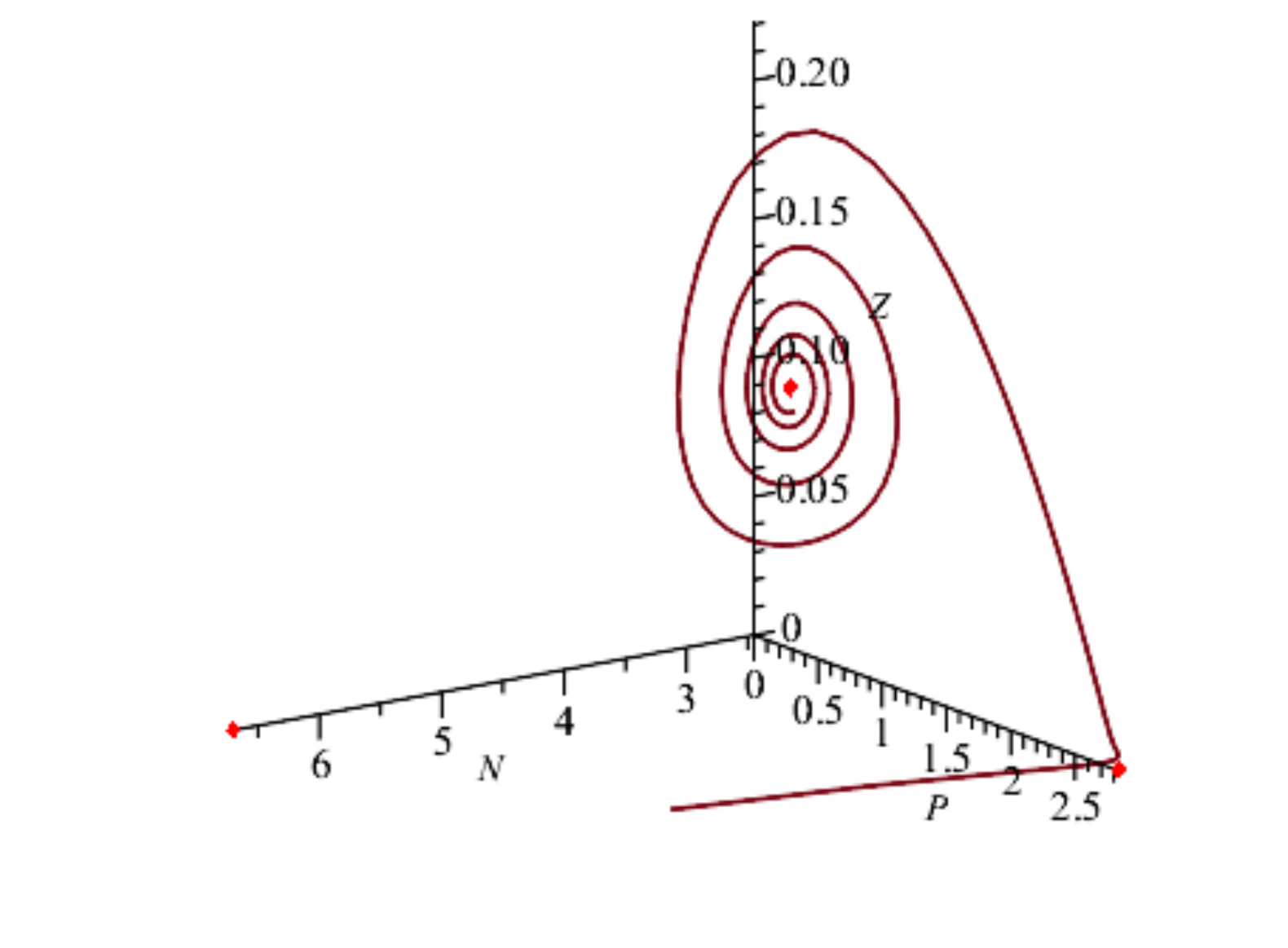}}
  \end{minipage}%
  \begin{minipage}[c]{0.45\linewidth}
\centering
     \fbox{\includegraphics[width=0.80\linewidth,viewport=20 360 475 695,clip=true]{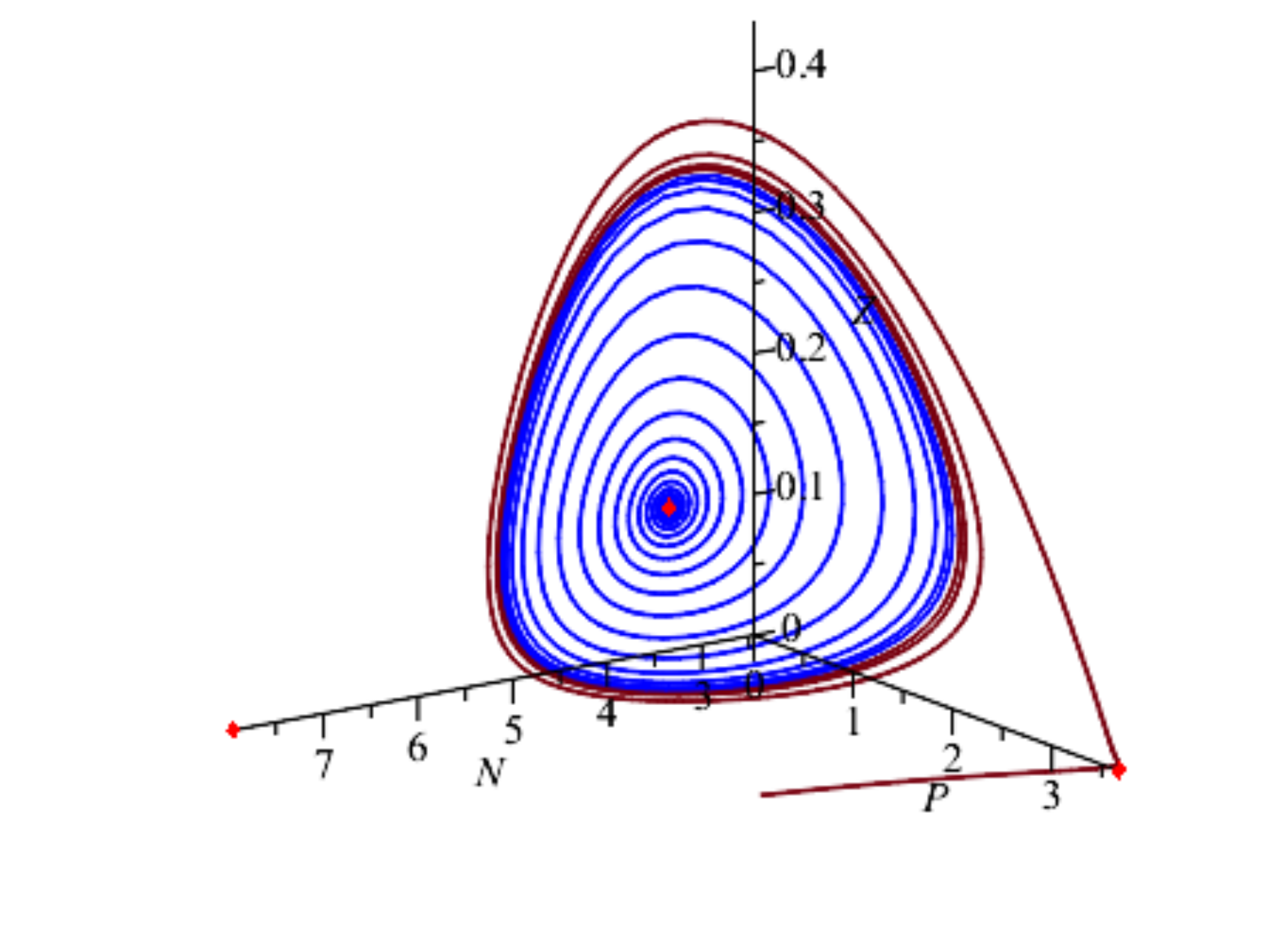}}
  \end{minipage}
\caption{Before and after Hopf bifurcation: $D{=}1$, $D_1{=}1.2$, and $D_2{=}1.1$ and using rate functions \eqref{HIIIdefs}} \label{h3Dsunequal}
\end{figure}
\newpage
\section{Uniform Approximation}
\label{Appendix}
For the proof of theorem \ref{stability_general}, we need lemma
\ref{uniformapproximation}, which states that, when $(D_1,D_2)$ is close to $(D,D)$,
$\gamma'(D_1,D_2)(\mu)$ is  uniformly approximated by $\gamma'(D,D)(\mu) = A'(\mu)$ 
on an interval $[\mu_{c_2}- \delta_0, \mu_{c_2}+ \delta_0]$, where $\mu_{c_2}$ is a point where
$A(\mu_{c_2}){=}0$ and $A'(\mu_{c_2}) > 0$. 

To recapitulate theorem \ref{stability_special},  the working assumptions are that a value $D$ is fixed,
 a coexistence equilibrium $E_2(\mu, D, D)$ exists, 
and that there is a range of parameters $\mu$ for
which the linearizations at the coexistence equilibrium 
have a complex conjugate pair of eigenvalues.  
Moreover, at the parameter value $\mu_{c_2}$, 
the linearization of the system has a purely imaginary pair of eigenvalues.
For any $\mu$ slightly smaller than $\mu_{c_2}$, the pair of complex eigenvalues
has a negative real part and for any $\mu$ slightly larger than $\mu_{c_2}$ the pair of complex
eigenvalues has a positive real part. 
By lemma
\ref{NZsmoothness}
there is an interval $I$ containing $\mu_{c_2}$ and a disc $\Delta$ centered at $(D,D)$
such that the functions $N(\mu, D_1, D_2)$ and $Z(\mu, D_1, D_2)$ defined on $I{\times}\Delta$
smoothly parametrize the locus of coexistence equilibria near $E_2(\mu, D,D)$.

Uniform approximation of $\gamma'(D_1, D_2)(\mu)$ by $\gamma'(D, D)(\mu)$ implies that, for the
coexistence equilibrium $E_2(\mu, D_1, D_2)$, there is a $\mu$-interval in $I$ for which the eigenvalues of
the linearizations exhibit the same qualitative behavior as described for
$E_2(\mu,D,D)$, as shown in theorem~\ref{stability_general}.
\begin{proposition}[Lemma \ref{uniformapproximation}]
  \label{convergenceestimate}
Assume $f_1$ is three times continuously differentiable and $f_2$ is two times
continuously differentiable. 
Then
there exists a $\mu$-interval $[\mu{-}\hat{\delta}, \mu{+}\hat{\delta}]$ on which the $\mu$-derivative $\gamma'(D_1,D_2)(\mu)$ is uniformly approximated by
$\gamma'(D,D)(\mu) = A'(\mu)$.
In fact, there exists a constant $C$ such that 
\begin{equation}
  \label{derivativebound}
  \abs{ \gamma'(D_1, D_2)( \mu) - \gamma'(D,D)(\mu)} \leq C \cdot \Dist{(D_1, D_2)}{(D, D)}
\end{equation}
for any $\mu \in [\mu_{c_2}{-}\hat{\delta}, \mu_{c_2}{+}\hat{\delta}]$.
\end{proposition}
The proposition follows from a sequence of lemmas and estimates, given below.  
In the course of proving these results, we find it necessary to impose the differentiability conditions
on $f_1$ and $f_2$. 

An essential ingredient in the process is to obtain bounds on magnitudes of the differences 
\begin{align*}
  p_0(\mu, D_1, D_2) - p_0(\mu, D, D),& &  p_1(\mu, D_1, D_2) - p_1(\mu, D,  D),& &
                             p_2(\mu, D_1, D_2) - p_2(\mu, D, D),
\intertext{and}
    p'_0(\mu, D_1, D_2) - p'_0(\mu, D, D),& & p'_1(\mu, D_1, D_2) - p'_1(\mu, D,  D),& &
      p'_2(\mu, D_1, D_2) - p'_2(\mu, D, D)
\end{align*}
in terms of $\Dist{(D_1, D_2)}{(D, D)}$,
 and where the $p_i(\mu, D_1, D_2)= -a_{3-i}(\mu, D_1, D_2)$, $0 \leq i \leq 2$ are given by 
the formulas \eqref{JE2coefficienta1}, \eqref{JE2coefficienta2}, and \eqref{JE2coefficienta3}.

We now explain the role played by these bounds.
By the chain rule, we compute $\gamma'(D_1,D_2)(\mu)$ as the inner product of
 a row vector $\nabla \gamma$ with a column vector $(p_0', p_1', p_2')$:
\begin{equation*}
\gamma'(D_1,D_2)(\mu) = \bigl \langle \nabla \gamma(p_0,p_1,p_2) , ( p_0', p_1', p_2') \bigr \rangle (\mu, D_1, D_2).
\end{equation*}
\begin{remark}
  In order to avoid extremely long expressions in the following analysis, 
 we use abbreviations such as
\begin{equation}  \label{abbreviation}
\bigl[D(M^{-1})\bigr](\mu, D_1,D_2)
\mathrel{\mathop :}=
  D(M^{-1})\bigl(p_0(\mu, D_1, D_2),p_1(D_1, D_2, \mu),p_2(\mu, D_1, D_2)\bigr).
\end{equation}
\end{remark}
We can write
\begin{align*}
[\nabla \gamma(p_0, p_1,p_2)](\mu, D_1, D_2) &= \pi_3 \circ  [D(M^{-1})(p_0, p_1,p_2)](\mu, D_1, D_2) 
\\
 &=
 \langle 0, 0, 1 \rangle  
\cdot
 \begin{bmatrix}
    \nabla \alpha (p_0, p_1,p_2)
\\
   \nabla \beta (p_0, p_1,p_2)
\\
   \nabla \gamma (p_0, p_1,p_2)
  \end{bmatrix}(\mu, D_1, D_2),
\end{align*}
where $\langle 0, 0, 1 \rangle$ represents the projection $\pi_3$ to the third coordinate 
$\gamma$ in $P_1(x)^-{\times}P_2(x)^-$.
The gradients  $\nabla \alpha$,  $\nabla \beta$, and $\nabla \gamma$ are evaluated at 
\begin{align*}
p_0(\mu, D_1, D_2) &= -a_3(\mu, D_1, D_2),
\\
p_1(\mu, D_1, D_2) &= -a_2(\mu, D_1, D_2),
\intertext{and}
p_2(\mu, D_1, D_2) &= -a_1(\mu, D_1, D_2),
\end{align*}
where explicit expressions for $a_i(\mu, D_1, D_2)$ are given in 
\eqref{JE2coefficienta1}, \eqref{JE2coefficienta2}, and \eqref{JE2coefficienta3}.
The derivatives $p_0'$, $p_1'$, and $p_2'$ are evaluated at $(\mu, D_1, D_2)$. 

Applying  the convention of \eqref{abbreviation}, we have
\begin{multline*}
  \gamma'(D_1, D_2)( \mu) - \gamma'(D,D)(\mu)
 \\
=  \bigl[\pi_3 \circ  D(M^{-1})(p_0, p_1, p_2)\cdot ( p_0', p_1', p_2')\bigr](\mu, D_1, D_2)
\\
   - \bigl[\pi_3 \circ D(M^{-1})(p_0, p_1, p_2)\cdot ( p_0',p_1', p_2' ) \bigr](\mu, D, D ).
\end{multline*}
Before we go farther, we compress taking the derivative $D(M^{-1})$  at $(p_0, p_1, p_2)$ and evaluating on the vector 
$( p_0', p_1', p_2' )$, writing
\begin{equation*}
   D(M^{-1})\cdot ( p_0', p_1', p_2' ) \mathrel{\mathop :}= D(M^{-1})(p_0, p_1, p_2)\cdot ( p_0', p_1', p_2').
\end{equation*}
Then the previous equation becomes
\begin{multline*}
  \gamma'(D_1, D_2)( \mu) - \gamma'(D,D)(\mu)
 \\
= \bigl[\pi_3 \circ  D(M^{-1}) \cdot ( p_0', p_1', p_2' )\bigr](\mu, D_1, D_2) 
       - \bigl[\pi_3 \circ D(M^{-1}) \cdot ( p_0',p_1', p_2' ) \bigr](\mu, D, D ).
\end{multline*}
We can estimate using operator norms computed in terms of the Euclidean metrics. 
\begin{multline}\label{gammaprimedifference}
 \abs{ \gamma'(D_1, D_2)( \mu) - \gamma'(D,D)(\mu)}
\\
=
 \abs{\bigl[ \pi_3 \circ  D(M^{-1})\cdot ( p_0', p_1', p_2' ) \bigr](\mu, D_1, D_2)
   - \bigl[\pi_3 \circ D(M^{-1})\cdot ( p_0',p_1', p_2') \bigr](\mu, D, D)}
\\
\leq \norm{\pi_3} \bignorm {\bigl[ D(M^{-1})\cdot ( p_0', p_1', p_2' ) \bigr](\mu, D_1, D_2)-\bigl[D(M^{-1})\cdot ( p_0',p_1', p_2') \bigr](\mu, D, D)}
\\
\leq  \bignorm { \bigl[D(M^{-1}) \cdot ( p_0', p_1', p_2' ) \bigr](\mu, D_1, D_2)-\bigl[D(M^{-1})\cdot ( p_0',p_1', p_2') \bigr](\mu, D, D)},
\end{multline}
since the norm of a projection is $1$.  Now we use the triangle inequality to bound the last expression.
\begin{multline} \label{gammaprimedifference1}
  \bignorm{ \bigl[ D(M^{-1})\cdot ( p_0', p_1', p_2' ) \bigr](\mu, D_1, D_2)-\bigl[D(M^{-1})\cdot ( p_0',p_1', p_2') \bigr](\mu, D, D)}
\\
  \leq \bignorm{\bigl[ D(M^{-1})\cdot ( p_0', p_1', p_2' ) \bigr](\mu, D_1, D_2)
    - \bigl[ D(M^{-1})\bigr](\mu, D_1, D_2)\cdot \bigl[ ( p_0', p_1', p_2' ) \bigr](\mu, D, D)}
\\
  + \bignorm{ \bigl[ D(M^{-1})\bigr](\mu, D_1, D_2)\cdot \bigl[ ( p_0', p_1', p_2' ) \bigr](\mu, D, D)
            - \bigl[D(M^{-1})\cdot ( p_0',p_1', p_2') \bigr](\mu, D, D) }.
\end{multline}
Explicit expansion of the first summand in \eqref{gammaprimedifference1} is given in the proof
of lemma \ref{firstsummand}, where we will see the role of the bounds on
\begin{equation*}
      p'_0(\mu, D_1, D_2) - p'_0(\mu, D, D), \quad  p'_1(\mu, D_1, D_2) - p'_1(\mu, D,  D), \quad
      p'_2(\mu, D_1, D_2) - p'_2(\mu, D, D).
\end{equation*}
Similarly, explicit expansion of the second summand in \eqref{gammaprimedifference1} is given in the proof
of lemma \ref{secondsummand}, where we will see the role of the bounds on
\begin{equation*}
  p_0(\mu, D_1, D_2) - p_0(\mu, D, D), \quad  p_1(\mu, D_1, D_2) - p_1(\mu, D,  D), \quad
      p_2(\mu, D_1, D_2) - p_2(\mu, D, D).
\end{equation*}
\begin{lemma}
  \label{firstsummand}
There is a constant $C_3$ such that the first summand in the expansion \eqref{gammaprimedifference1}
satisfies
\begin{multline}
  \label{firstsummandbound}
  \bignorm{\bigl[ D(M^{-1})\cdot ( p_0', p_1', p_2' ) \bigr](\mu, D_1, D_2)
    - \bigl[ D(M^{-1})\bigr](\mu, D_1, D_2)\cdot\bigl[ ( p_0', p_1', p_2' ) \bigr](\mu, D, D)}
\\
\leq  C_3 \, \cdot \, \Dist{(D_1, D_2)}{(D,D)}.
\end{multline}
\end{lemma}
\begin{proof}
By the basic property of the operator norm, 
\begin{multline*}
  \bignorm{\bigl[ D(M^{-1})\cdot( p_0', p_1', p_2' ) \bigr](\mu, D_1, D_2)
    - \bigl[ D(M^{-1})\bigr](\mu, D_1, D_2)\cdot\bigl[ ( p_0', p_1', p_2' ) \bigr](\mu, D, D)}
\\
\leq \bignorm{\bigl[ D(M^{-1})\bigr](\mu, D_1, D_2)  }\cdot 
   \bignorm{\bigl[( p_0', p_1', p_2' ) \bigr](\mu, D_1, D_2) - \bigl[ ( p_0', p_1', p_2' ) \bigr](\mu, D, D)  }
\end{multline*}
Examining the factor coming from the operator norm, 
\begin{multline}  \label{firstfactorfirstsummand}
  \bignorm{\bigl[ D(M^{-1})\bigr](\mu, D_1, D_2)}
\\
  = \bignorm{D(M^{-1})\bigl(p_0(\mu, D_1, D_2), p_1(\mu, D_1, D_2), p_2(D_1,D_2, \mu)\bigr)}\leq C_{DM^{-1}},
\end{multline}
for some constant $C_{DM^{-1}}$.  This is because, as $(D_1,D_2)$ ranges over any closed disc centered on $(D,D)$ and $\mu$
ranges over any closed interval containing $\mu_{c_2}$,
the coordinates $(p_0, p_1, p_2)$ are contained in a compact set, so there is a constant $C_{DM^{-1}}$ that bounds the norm
of $D(M^{-1})$ at any of these points. 

Now we require a bound on the other factor, for which we have
\begin{multline}  \label{secondfactorfirstsummand}
   \bignorm{[( p_0', p_1', p_2') ](\mu, D_1, D_2)-  [(p_0', p_1', p_2') ](\mu, D, D)}^2
\\
 = \bigabs{p_0'(\mu, D_1, D_2){-}p_0'(\mu, D, D)}^2 
            \!+\! \bigabs{p_1'(\mu, D_1, D_2){-}p_1'(\mu, D, D)}^2 
                          \!+\! \bigabs{p_2'(\mu, D_1, D_2){-}p_2'(\mu, D, D)}^2
\\
\leq \bigl((C_0')^2{+}(C_1')^2{+}(C_2')^2\bigr)\,\cdot\,\Dist{(D_1, D_2)}{(D, D)}^2,
\end{multline}
provided by combining the results of propositions \ref{p0bounds}, \ref{p1bounds}, and \ref{p2bounds}.
Combining this bound with the bound
on the expression \eqref{firstfactorfirstsummand}, the first summand in \eqref{gammaprimedifference1} is
bounded by a constant times $\Dist{(D_1, D_2)}{(D, D)}$.
\end{proof}
Now we bound the second summand in \eqref{gammaprimedifference1}.
\begin{lemma}
  \label{secondsummand}
There is a constant $C_4$ such that the second summand in the expansion \eqref{gammaprimedifference1} satisfies
\begin{multline}
  \label{secondsummandbound}
 \bignorm{ \bigl[ D(M^{-1})\bigr](\mu, D_1, D_2)\cdot\bigl[ ( p_0', p_1', p_2' ) \bigr](\mu, D, D)
- \bigl[D(M^{-1})\cdot( p_0',p_1', p_2') \bigr](\mu, D, D) }
\\
\leq
C_4 \, \cdot \, \Dist{(D_1, D_2)}{(D,D)}.
\end{multline}
\end{lemma}
\begin{proof}
To handle the second summand in \eqref{gammaprimedifference1}, 
we have the bound
\begin{multline} \label{gammaprimesecondsummand}
  \bignorm{ \bigl[ D(M^{-1})\bigr](\mu, D_1, D_2)\cdot\bigl[ ( p_0', p_1', p_2' ) \bigr](\mu, D, D)
- \bigl[D(M^{-1})\cdot( p_0',p_1', p_2') \bigr](\mu, D, D) }
\\
\leq
 \bignorm{ \bigl[ D(M^{-1})\bigr](\mu, D_1, D_2)- \bigl[D(M^{-1})\bigr](D,D,\mu)} \bignorm{\bigl[ ( p_0',p_1', p_2') \bigr](\mu, D, D) }.
\end{multline}
To the first factor in the bounding term, apply the mean value theorem 
\cite[p.103, Corollary~1]{LangAnalysisII} for vector-valued functions of several variables, obtaining 
\begin{multline*}
  \bignorm{[D(M^{-1})](\mu, D_1, D_2)- [D(M^{-1})](D,D,\mu)}
\\
 \leq \bignorm{[D(D(M^{-1}))](q_0,q_1,q_2)}\cdot\bignorm{[(p_0,p_1, p_2)](\mu, D_1, D_2)-[(p_0,p_1, p_2)](\mu, D, D)}.
\end{multline*}
The derivative of the map $P_3(x)^- \rightarrow M_{3,3}(\bR)$, $(p_0,p_1,p_2) \mapsto D(M^{-1})(p_0,p_1,p_2)$ is continuous,
because $M^{-1}$ is $C^{\infty}$.  The evaluation point
$(q_0, q_1,q_2)$ is on the line segment connecting 
$[(p_0,p_1, p_2)](\mu, D_1, D_2)$ and $[(p_0,p_1, p_2)](\mu, D, D)$.  
Again the possibilities range over
a compact set, so the norm of the second derivative satifies
\begin{equation}
  \label{firstfactorbound}
  \bignorm{[D(D(M^{-1}))](q_0,q_1,q_2)} \leq C_{D^2(M^{-1})}, 
\end{equation}
for a  constant $C_{D^2(M^{-1})}$  independent of $(D_1,D_2, \mu)$.
We compute
\begin{multline} \label{secondfactorbound}
  \bignorm{p(\mu, D_1, D_2)-p(\mu, D, D)}^2 =
\\
 \bigabs{p_0(\mu, D_1, D_2){-}p_0(\mu, D, D)}^2 + \bigabs{p_1(\mu, D_1, D_2){-}p_1(\mu, D, D)}^2 + \bigabs{p_2(\mu, D_1, D_2){-}p_2(\mu, D, D)}^2
\\
\leq \bigl((C_0)^2{+}(C_1)^2{+}(C_2)^2\bigr) \cdot \Dist{(D_1, D_2)}{(D, D)}^2,
\end{multline}
 by combining the results of propositions \ref{p0bounds}, \ref{p1bounds}, and \ref{p2bounds}.
Assembling the bounds in \eqref{firstfactorbound} and \eqref{secondfactorbound}, 
$\bignorm{[D(M^{-1})](\mu, D_1, D_2)- [D(M^{-1})](D,D,\mu)}$ is bounded by 
a constant times $\Dist{(D_1, D_2)}{(D, D)}$.  This takes care of the first factor on the righthand side
of \eqref{gammaprimesecondsummand}.

For the remaining  factor in the bounding term in  \eqref{gammaprimesecondsummand},  the norm of the tangent vector satisfies
\begin{equation}  \label{thirdfactorbound}
  \bignorm{[( p_0', p_1', p_2') ](\mu, D, D)} \leq C_T,
\end{equation}
for some constant $C_T$, independent of $\mu$.
Now that we have taken care of both factors on the righthand side of \eqref{gammaprimesecondsummand}
the second summand in \eqref{gammaprimedifference1} is bounded by
a constant times
$\Dist{(D_1, D_2)}{(D, D)}$, as claimed.  
\end{proof}
We can now prove the uniform convergence result.
\begin{proof}
  [Proof of proposition \ref{convergenceestimate}.]
Combining the inequalities \eqref{gammaprimedifference} and \eqref{gammaprimedifference1}, we have
\begin{multline*}
  \abs{ \gamma'(D_1, D_2)( \mu) - \gamma'(D,D)(\mu)} 
\\
\leq
\bignorm{\bigl[ D(M^{-1})\cdot ( p_0', p_1', p_2' ) \bigr](\mu, D_1, D_2)
    - \bigl[ D(M^{-1})\bigr](\mu, D_1, D_2)\cdot \bigl[ ( p_0', p_1', p_2' ) \bigr](\mu, D, D)}
\\
  + \bignorm{ \bigl[ D(M^{-1})\bigr](\mu, D_1, D_2)\cdot \bigl[ ( p_0', p_1', p_2' ) \bigr](\mu, D, D)
            - \bigl[D(M^{-1})\cdot ( p_0',p_1', p_2') \bigr](\mu, D, D) }
\end{multline*}
Using the observations detailed in lemmas \ref{firstsummand} and \ref{secondsummand}, 
\begin{equation*}
  \abs{ \gamma'(D_1, D_2)( \mu) - \gamma'(D,D)(\mu)} 
\leq
(C_3 + C_4) \, \cdot \, \Dist{(D_1, D_2)}{(D,D)}
\end{equation*}
Thus, there is a constant $C$ such that
\begin{equation*}
  \abs{ \gamma'(D_1, D_2)( \mu) - \gamma'(D,D)(\mu)} \leq C\cdot \Dist{(D_1, D_2)}{(D,D)}
\end{equation*}
for any $\mu \in [\mu_{c_2}- \delta_0, \mu_{c_2}+ \delta_0]$.
\end{proof}
We have already used a technique of obtaining bounds by splitting quantities. As we will continue to exploit the technique
in the following results, we formulate the lemma \ref{elementary} for reference. 
\begin{lemma}
  \label{elementary}
Let $Q_1(\mu, D_1, D_2)$ and $Q_2(\mu, D_1, D_2)$ be quantities defined on a domain $I{\times}\Delta$ satisfying the following conditions.
\begin{enumerate}
\item There is a constant $c_1$ such that
  \begin{equation*}
  \abs{Q_1(\mu, D_1, D_2) - Q_1(\mu, D, D)} \leq c_1\Dist{(D_1, D_2)}{(D, D)}.  
  \end{equation*}
\item  There is a constant $c_2$ such that 
  \begin{equation*}
    \abs{Q_2(\mu, D_1, D_2) - Q_2(\mu, D, D)} \leq c_2\Dist{(D_1, D_2)}{(D, D)}.
  \end{equation*}
\item  There are constants $c_3$ and $c_4$ such that $\abs{Q_1(\mu, D_1, D_2)} \leq c_3$ 
for $(\mu, D_1, D_2) \in I{\times}\Delta$ and for $\Dist{(D_1, D_2)}{(D, D)}$ sufficiently small, 
and $\abs{Q_2(\mu,D,D)}\leq c_4$ for $\mu \in I$.
\end{enumerate}
Then there is a constant $c_5$ such that, for  $\Dist{(D_1, D_2)}{(D, D)}$ sufficiently small, 
\begin{equation*}
  \abs{Q_1(\mu, D_1, D_2)\cdot Q_2(\mu, D_1, D_2) - Q_1(\mu, D, D)\cdot Q_2(\mu, D, D)} 
   \leq c_5 \Dist{(D_1, D_2)}{(D, D)}  \qed
\end{equation*} 
\end{lemma}
As has been seen, the proof of proposition \ref{convergenceestimate} depends on the following three propositions. 
\begin{proposition}
  \label{p0bounds}
There are constants $C_0$ and $C_0'$ such that 
\begin{align}
  \abs{p_0(\mu, D_1, D_2) - p_0(\mu, D, D)} &\leq C_0\cdot \Dist{(D_1, D_2)}{(D, D)}, \label{p0bound}
\\
 \abs{p_0'(\mu, D_1, D_2) - p_0'(\mu, D, D)} &\leq C_0'\cdot \Dist{(D_1, D_2)}{(D, D)}. \label{p0primebound}
\end{align}
\end{proposition}
\begin{proposition}
  \label{p1bounds}
There are constants $C_1$ and $C_1'$ such that 
\begin{align}
  \abs{p_1(\mu, D_1, D_2) - p_1(\mu, D, D)} &\leq C_1\cdot \Dist{(D_1, D_2)}{(D, D)}, \label{p1bound}
\\
 \abs{p_1'(\mu, D_1, D_2) - p_1'(\mu, D, D)} &\leq C_1'\cdot \Dist{(D_1, D_2)}{(D, D)}, \label{p1primebound}
\end{align}
\end{proposition}
\begin{proposition}
  \label{p2bounds}
There are constants $C_2$ and $C_2'$ such that 
\begin{align}
  \abs{p_2(\mu, D_1, D_2) - p_2(\mu, D, D)} &\leq C_2\cdot \Dist{(D_1, D_2)}{(D, D)}, \label{p2bound}
\\
 \abs{p_2'(\mu, D_1, D_2) - p_2'(\mu, D, D)} &\leq C_2'\cdot \Dist{(D_1, D_2)}{(D, D)}, \label{p2primebound}
\end{align}
\end{proposition}
The proofs of propositions \ref{p0bounds}, \ref{p1bounds}, and \ref{p2bounds} depend in turn on a number of elementary bounds 
and estimates, given below in lemma \ref{f2primebound} and propositions \ref{DDbounds}, \ref{muD1D2bounds}, and \ref{constituentbounds}.
We give the quick proofs of lemma \ref{f2primebound} and propositions \ref{DDbounds} and \ref{muD1D2bounds}, because they are quite
short, postponing the proof of the many parts of proposition \ref{constituentbounds} to the end of the section.  After we state
these results, we prove propositions \ref{p0bounds}, \ref{p1bounds}, and \ref{p2bounds}.
\begin{lemma} \label{f2primebound}
  Given $D$, there is an interval $J$ containing $\lambda_Z(D)$ such that 
  \begin{equation*}
    f_2'(P) > f_2'\bigl(\lambda_Z(D)\bigr)/2 > 0
  \end{equation*}
for all $P \in J$.  Thus, for all $P \in J$, $f_2'(P)$ is bounded away from zero,
and, for all $D_2$ in the preimage $\lambda_Z^{-1}(J)$, $f_2'\bigl(\lambda_Z(D_2)\bigr)$ is bounded away from zero.
\end{lemma}
\begin{proof}
  By assumption on $f_2$, $f_2'\bigl(\lambda_Z(D)\bigr) > 0$.  By continuity of $f_2'$, there is an interval $J$
containing $\lambda_Z(D)$ such that, for all $P \in J$, 
\begin{equation*}
 -f_2'\bigl(\lambda_Z(D)\bigr)/2 <  f_2'(P) - f_2'\bigl(\lambda_Z(D)\bigr) < f_2'\bigl(\lambda_Z(D)\bigr)/2.  \qedhere
\end{equation*}
\end{proof}
\begin{proposition}
  \label{DDbounds}
Each of the quantities
\begin{gather*}
  f_1\bigl(N(\mu, D, D)\bigr), \quad f_1'\bigl(N(\mu, D, D)\bigr), \quad Z(\mu, D, D), 
\\
\quad  f_1''\bigl(N(\mu, D, D)\bigr), \quad N'(\mu, D, D),  \quad \text{and} \quad Z'(\mu, D, D)
\end{gather*}
is bounded by some constant on the interval $I$.
\end{proposition}
\begin{proof}
  Each of the listed functions is continuous on the closed interval $I$, so each one is bounded.
\end{proof}
\begin{proposition}  \label{muD1D2bounds}
  Each of the quantities 
  \begin{gather*}
    \lambda_Z(D_2) \quad f_1\bigl(N(\mu, D_1, D_2)\bigr) \quad f_1'\bigl(N(\mu, D_1, D_2)\bigr) \quad Z(\mu, D_1, D_2)
\\
  f_2'\bigl(\lambda_Z(D_2)\bigr) \quad f_1''\bigl(N(\mu, D_1, D_2)\bigr) \quad N'(\mu, D_1, D_2) \quad \text{and} \quad Z'(\mu, D_1, D_2)
  \end{gather*}
is bounded by some constant on the domain $I {\times} \Delta$.
\end{proposition}
\begin{proof}
  Each of the listed functions is continuous on the compact set $I{\times}\Delta$, so each one is bounded.
\end{proof}
The proof of the next result depends on many more details of system \eqref{NPZsys} and consequences drawn from them.  
Some steps in the proof are quite lengthy, so we postpone these details to the end of the section.
\begin{proposition}
  \label{constituentbounds}
Assume that the domain $I{\times}\Delta$ for which $N(\mu,D_1,D_2)$ and $Z(\mu,D_1,D_2)$ are
defined is a subset of $I{\times}\bR{\times}\lambda_Z^{-1}(J)$, where $J$ is as in lemma \ref{f2primebound}.
Then for $\mu \in I$ and $(D_1, D_2) \in \Delta$, these differences are bounded by constants times $\Dist{(D_1, D_2)}{(D, D)}$.
\begin{align*}
&\text{1.} \; \lambda_Z(D) - \lambda_Z(D_2). 
&\text{2.} \;  &f_1\bigl(N(\mu, D_1, D_2)\bigr) - f_1\bigl(N(\mu, D, D)\bigr).
\\
&\text{3.} \; f_1'\bigl(N(\mu, D_1, D_2)\bigr) - f_1'\bigl(N(\mu, D, D)\bigr). 
&\text{4.} \; &Z(\mu, D_1, D_2) - Z(\mu, D, D).
\\
&\text{5.} \; f_2'\bigl(\lambda_Z(D_2)\bigr) - f_2'\bigl(\lambda_Z(D)\bigr). 
&\text{6.} \; &f_1^{(2)}\bigl(N(\mu, D_1, D_2)\bigr) - f_1^{(2)}\bigl(N(\mu, D, D)\bigr).
\intertext{Moreover,}
&\text{7.}\; N'(\mu, D_1, D_2) - N'(\mu, D, D),   &\text{8.}\;  &Z'(\mu, D_1, D_2) - Z'(\mu, D, D),
\end{align*}
where the derivatives are taken with respect to $\mu$, are also bounded by constants times $\Dist{(D_1, D_2)}{(D, D)}$.
\end{proposition}
\begin{proof}[Proof of proposition \ref{p0bounds}]
After some reorganization, we have from \eqref{JE2coefficienta3}
\begin{multline}  \label{p0difference}
  p_0(\mu, D_1, D_2) - p_0(D,D,\mu) = -a_3(\mu, D_1, D_2) + a_3(D,D,\mu) 
\\
    =  -D_2 f_2'\bigl(\lambda_Z(D_2)\bigr) Z(\mu, D_1, D_2)\Bigl( D + \lambda_Z(D_2) f_1'\bigl(N(\mu, D_1, D_2)\bigr)\Bigr) 
\\
                    + D f_2'\bigl(\lambda_Z(D)\bigr) Z(\mu, D, D)\Bigl( D + \lambda_Z(D) f_1'\bigl(N(\mu, D, D)\bigr) \Bigr)
\\
   = \Bigl( D f_2'\bigl(\lambda_Z(D)\bigr) Z(\mu, D, D) -D_2 f_2'\bigl(\lambda_Z(D_2)\bigr) Z(\mu, D_1, D_2)\Bigr) D
\\
             + \Bigl[ D f_2'\bigl(\lambda_Z(D)\bigr) Z(\mu, D, D) \lambda_Z(D) f_1'\bigl(N(\mu, D, D)\bigr) 
\\
                   - D_2 f_2'\bigl(\lambda_Z(D_2)\bigr) Z(\mu, D_1, D_2) \lambda_Z(D_2) f_1'\bigl(N(\mu, D_1, D_2)\bigr)\Bigr].
\end{multline}
To bound $\abs{p_0(\mu, D_1, D_2) - p_0(D,D,\mu)}$, we bound the absolute values of summands in \eqref{p0difference} as follows. 
First, note $\abs{D{-}D_2} \leq \Dist{(D_1, D_2)}{(D, D)}$, so we bound
\begin{equation*}
\abs{  D f_2'\bigl(\lambda_Z(D)\bigr) Z(\mu, D, D) - D_2 f_2'\bigl(\lambda_Z(D_2)\bigr) Z(\mu, D_1, D_2)}
\end{equation*}
by using lemma \ref{elementary}, lemma \ref{DDbounds}, and lemma \ref{muD1D2bounds} 
to combine the noted bound with bounds 4 and 5 from proposition \ref{constituentbounds}; bound
\begin{multline*}
  \abs{D f_2'\bigl(\lambda_Z(D)\bigr) Z(\mu, D, D) \lambda_Z(D) f_1'\bigl(N(\mu, D, D)\bigr) 
                  \\
              - D_2 f_2'\bigl(\lambda_Z(D_2)\bigr) Z(\mu, D_1, D_2) \lambda_Z(D_2) f_1'\bigl(N(\mu, D_1, D_2)\bigr)}
\end{multline*}
using lemma \ref{elementary} to  combine bounds 1, 3, 4 and 5 with $\abs{D{-}D_2} \leq \Dist{(D_1, D_2)}{(D, D)}$.
We compute from \eqref{p0difference}
\begin{multline*}
  p_0'(\mu, D_1, D_2) - p_0'(\mu, D, D)
\\ 
= \Bigl( D f_2'\bigl(\lambda_Z(D)\bigr) Z'(\mu, D, D) -D_2 f_2'\bigl(\lambda_Z(D_2)\bigr) Z'(\mu, D_1, D_2)\Bigr) D
\\
\shoveleft{\quad \quad \quad + \Bigl[ D f_2'\bigl(\lambda_Z(D)\bigr) Z'(\mu, D, D) \lambda_Z(D) f_1'\bigl(N(\mu, D, D)\bigr)}
\\
            - D_2 f_2'\bigl(\lambda_Z(D_2)\bigr) Z'(\mu, D_1, D_2) \lambda_Z(D_2) f_1'\bigl(N(\mu, D_1, D_2)\bigr)\bigr]
\\
              +\bigl[ D f_2'\bigl(\lambda_Z(D)\bigr) Z(\mu, D, D) \lambda_Z(D) f_1^{(2)}\bigl(N(\mu, D, D)\bigr) N'(\mu, D, D)
\\
                 - D_2 f_2'\bigl(\lambda_Z(D_2)\bigr) Z(\mu, D_1, D_2) \lambda_Z(D_2) f_1^{(2)}\bigl(N(\mu, D_1, D_2)\bigr)N'(\mu, D_1, D_2) \Bigr].
\end{multline*}  
For $\abs{p_0'(\mu, D_1, D_2) - p_0'(D,D,\mu)}$, we bound from the first line of the expansion
\begin{equation*} 
  \abs{D f_2'\bigl(\lambda_Z(D)\bigr) Z'(\mu, D, D) - D_2 f_2'\bigl(\lambda_Z(D_2)\bigr) Z'(\mu, D_1, D_2)}
\end{equation*}
by combining bounds 5 and 8 with the bound $\abs{D{-}D_2} \leq \Dist{(D_1, D_2)}{(D, D)}$; bound from the second and third lines
\begin{multline*}
\abs{D f_2'\bigl(\lambda_Z(D)\bigr) Z'(\mu, D, D) \lambda_Z(D) f_1'\bigl(N(\mu, D, D)\bigr) 
            \\- D_2 f_2'\bigl(\lambda_Z(D_2)\bigr) Z'(\mu, D_1, D_2) \lambda_Z(D_2) f_1'\bigl(N(\mu, D_1, D_2)\bigr)}
  \end{multline*}
by combining bounds 1, 3, 5, and 8 with the bound $\abs{D{-}D_2} \leq \Dist{(D_1, D_2)}{(D, D)}$; bound from the fourth and fifth lines
\begin{multline*}
\abs{D f_2'\bigl(\lambda_Z(D)\bigr) Z(\mu, D, D) \lambda_Z(D) f_1^{(2)}\bigl(N(\mu, D, D)\bigr) N'(\mu, D, D)
             \\- D_2 f_2'\bigl(\lambda_Z(D_2)\bigr) Z(\mu, D_1, D_2) \lambda_Z(D_2) f_1^{(2)}\bigl(N(\mu, D_1, D_2)\bigr)N'(\mu, D_1, D_2)}  
\end{multline*}
by combining bounds 1, 4, 5, 6, and 7 from proposition \ref{constituentbounds} 
with the bound $\abs{D{-}D_2} \leq \Dist{(D_1, D_2)}{(D, D)}$.
\end{proof}
\begin{proof}[Proof of proposition \ref{p1bounds}]
After some reorganization, placing terms belonging to $p_1(\mu, D_1, D_2)$ down the left side of the display, 
we have from \eqref{JE2coefficienta2}
\begin{multline}
 \label{p1difference}
 p_1(\mu, D_1, D_2) - p_1(D,D,\mu)  =  -a_2(\mu, D_1, D_2) + a_2(D,D,\mu)
\\
 \shoveleft{\quad =-\lambda_{Z}(D_2) Z(\mu, D_1, D_2) f_1'\bigl(N(\mu, D_1, D_2)\bigr) f_2'\bigl(\lambda_{Z}(D_2)\bigr)}
\\
\shoveright{+\lambda_{Z}(D) Z(\mu, D, D) f_1'\bigl(N(\mu, D, D)\bigr) f_2'\bigl(\lambda_{Z}(D)\bigr)}
\\
\qquad \quad - D_2Z(\mu, D_1, D_2) f_2'\bigl(\lambda_{Z}(D_2)\bigr) +  D Z(\mu, D, D) f_2'\bigl(\lambda_{Z}(D)\bigr)
\\
\qquad \qquad  - D Z(\mu, D_1, D_2) f_2'\bigl(\lambda_{Z}(D_2)\bigr) + D Z(\mu, D, D) f_2'\bigl(\lambda_{Z}(D)\bigr)
\\
\qquad \qquad \quad - D_{1}\lambda_{Z}(D_2)f_1'\bigl(N(\mu, D_1, D_2)\bigr) + D \lambda_{Z}(D)f_1'\bigl(N(\mu, D, D)\bigr)
\\
\shoveright{ + D \gamma_{1}f_{1}\bigl(N(\mu, D_1, D_2)\bigr)  - D \gamma_{1}f_{1}(N(\mu, D, D)\bigr)}
\\
- D D_{1} + D^2. 
\end{multline}
To bound $\abs{p_1(\mu, D_1, D_2) - p_1(D,D,\mu)}$ by a constant multiple of $\Dist{(D_1, D_2)}{(D,D)}$, 
first observe that both $\abs{D{-}D_1}$ and $\abs{D{-}D_2}$ are bounded by $\Dist{(D_1, D_2)}{(D,D)}$.
Similarly bound 
\begin{multline*}
\abs{-\lambda_{Z}(D_2) Z(\mu, D_1, D_2) f_1'\bigl(N(\mu, D_1, D_2)\bigr) f_2'\bigl(\lambda_{Z}(D_2)\bigr) 
\\
          +    \lambda_{Z}(D) Z(\mu, D, D) f_1'\bigl(N(\mu, D, D)\bigr) f_2'\bigl(\lambda_{Z}(D)\bigr)}
\end{multline*}
by combining bounds 1, 4, 3, and 5 from proposition \ref{constituentbounds}; bound
\begin{align*}
   \abs{-D_2 Z(\mu, D_1, D_2) f_2'\bigl(\lambda_{Z}(D_2)\bigr) &+   D Z(\mu, D, D) f_2'\bigl(\lambda_{Z}(D)\bigr)}
\intertext{by combining bounds 4 and 5 with  $\abs{D{-}D_2} \leq \Dist{(D_1, D_2)}{(D,D)}$;  bound}
\abs{- D Z(\mu, D_1, D_2) f_2'\bigl(\lambda_{Z}(D_2)\bigr) &+ D Z(\mu, D, D) f_2'\bigl(\lambda_{Z}(D)\bigr)}
\intertext{by combining bounds 4 and 5; bound}
\abs{- D_1\lambda_{Z}(D_2)f_1'\bigl(N(\mu, D_1, D_2)\bigr) &+ D \lambda_{Z}(D)f_1'\bigl(N(\mu, D, D)\bigr)}
\intertext{by combining bounds 1 and 3 with $\abs{D{-}D_1}\leq \Dist{(D_1, D_2)}{(D,D)}$; and}
\abs{D \gamma_{1}f_{1}\bigl(N(\mu, D_1, D_2)\bigr) &- D \gamma_{1}f_{1}\bigl(N(\mu, D, D)\bigr)}
\end{align*}
is taken care of in bound 2 of proposition \ref{constituentbounds}.  Finally, $\abs{-DD_1 + D^2}$ has
already been taken care of.  Adding all these bounds,
$\abs{p_1(\mu, D_1, D_2) - p_1(D,D,\mu)}$ is bounded by a constant multiple of $\Dist{(D_1, D_2)}{(D,D)}$.

Rather than exhibit a complete formula for  $p_1'(\mu, D_1, D_2) - p_1'(D,D,\mu)$,
we pick apart equation \eqref{p1difference} to express this difference as a sum of expressions.
From the first two lines, 
\begin{multline} \label{firsttwolines}
-\lambda_{Z}(D_2) Z'(\mu, D_1, D_2) f_1'\bigl(N(\mu, D_1, D_2)\bigr) f_2'\bigl(\lambda_{Z}(D_2)\bigr) 
\\
            -\lambda_{Z}(D_2) Z(\mu, D_1, D_2) f_1^{(2)}\bigl(N(\mu, D_1, D_2)\bigr) N'(\mu, D_1, D_2) f_2'\bigl(\lambda_{Z}(D_2)\bigr)
\\
+\lambda_{Z}(D) Z'(\mu, D, D) f_1'\bigl(N(\mu, D, D)\bigr) f_2'\bigl(\lambda_{Z}(D)\bigr) 
\\
+\lambda_{Z}(D) Z(\mu, D, D) f_1^{(2)}\bigl(N(\mu, D, D)\bigr) N'(\mu, D, D) f_2'\bigl(\lambda_{Z}(D)\bigr)
\end{multline}
is involved in the sum. From lines three through six, the expressions involved are
\begin{gather}
   - D_2Z'(\mu, D_1, D_2) f_2'\bigl(\lambda_{Z}(D_2)\bigr) +  DZ'(\mu, D, D) f_2'\bigl(\lambda_{Z}(D)\bigr),  \label{linethree}
\\
- D Z'(\mu, D_1, D_2) f_2'\bigl(\lambda_{Z}(D_2)\bigr) + D Z'(\mu, D, D) f_2'\bigl(\lambda_{Z}(D)\bigr), \label{linefour}
\\
- D_{1}\lambda_{Z}(D_2)f_1^{(2)}\bigl(N(\mu, D_1, D_2)\bigr)N'(\mu, D_1, D_2) 
              + D \lambda_{Z}(D)f_1^{(2)}\bigl(N(\mu, D, D)\bigr)N'(\mu, D, D), \label{linefive}
\\
 D \gamma_{1}f_{1}'\bigl(N(\mu, D_1, D_2)\bigr)N'(\mu, D_1, D_2)  - D \gamma_{1}f_{1}'\bigl(N(\mu, D, D)\bigr)N'(\mu, D, D). \label{linesix}
\end{gather}
Making several applications of lemma \ref{elementary}, lemma \ref{DDbounds}, lemma \ref{muD1D2bounds}, and proposition \ref{constituentbounds},
we find that the absolute value of each of the quantities displayed in 
\eqref{firsttwolines}, \eqref{linethree}, \eqref{linefour}, \eqref{linefive}, and \eqref{linesix} is bounded by a constant times
$\Dist{(D_1, D_2)}{(D, D)}$. Consequently, there is a constant $C_1'$ such that
\begin{equation*} 
  \abs{ p_1'(\mu, D_1, D_2) - p_1'(D,D,\mu)  }  \leq C_1' \cdot \Dist{(D_1, D_2)}{(D, D)}. \qedhere
\end{equation*}  
\end{proof}
\begin{proof}[Proof of proposition \ref{p2bounds}]
Using the formula \eqref{JE2coefficienta1} and organizing the difference $p_2(\mu, D_1, D_2) - p_2(\mu, D, D)$ to display
terms belonging to $p_2(\mu, D_1, D_2)$ down the left side of the display, we have
\begin{multline} \label{p2difference}
  p_2(\mu, D_1, D_2) - p_2(\mu, D, D) = -a_1(D_1 , D_2, \mu) + a_1(\mu, D, D) 
  \\
= - Z(\mu, D_1, D_2) f_2'\bigl(\lambda_{Z}(D_2)\bigr) +  Z(\mu, D, D) f_2'\bigl(\lambda_{Z}(D)\bigr)
\\
\qquad  \quad -\lambda_{Z}(D_2)f_1'\bigl(N(\mu, D_1, D_2)\bigr) + \lambda_{Z}(D)f_1'\bigl(N(\mu, D, D)\bigr)
 \\ 
\qquad \qquad  +\gamma_{1}f_{1}\bigl(N(\mu, D_1, D_2)\bigr)- \gamma_{1}f_{1}\bigl(N(\mu, D, D)\bigr)
\\
 - D_{1} + D.    \qquad \qquad \qquad         
\end{multline}
The bound on $\abs{p_2(\mu, D_1, D_2) - p_2(\mu, D, D)}$ in \eqref{p2bound} follows from bounds on summands in \eqref{p2difference}, as follows.
Apply lemma \ref{elementary}, lemma \ref{DDbounds}, lemma \ref{muD1D2bounds}, and bounds 4 and 5 from proposition \ref{constituentbounds}
to bound the term 
\begin{align*}
  \bigabs{ - Z(\mu, D_1, D_2) f_2'\bigl(\lambda_{Z}(D_2)\bigr) &+  Z(\mu, D, D) f_2'\bigl(\lambda_{Z}(D)\bigr)}; 
\intertext{bound}
\bigabs{-\lambda_{Z}(D_2)f_1'\bigl(N(\mu, D_1, D_2)\bigr) &+ \lambda_{Z}(D)f_1'\bigl(N(\mu, D, D)\bigr)}
\intertext{in the same manner, using bounds 1 and 3 from proposition \ref{constituentbounds}. Then bound}
 \bigabs{ \gamma_{1}f_{1}\bigl(N(\mu, D_1, D_2)\bigr) &- \gamma_{1}f_{1}\bigl(N(\mu, D, D)\bigr)}
\end{align*}
 using bound 2 from proposition \ref{constituentbounds}.
Finally, $\abs{ - D_1 + D}$ is bounded by $\Dist{(D_1, D_2)}{(D, D)}$.

From \eqref{p2difference}, 
\begin{multline} \label{p2primedifference}
   p_2'(\mu, D_1, D_2) - p_2'(\mu, D, D) = 
\\
- Z'(\mu, D_1, D_2) f_2'\bigl(\lambda_{Z}(D_2)\bigr) +  Z'(\mu, D, D) f_2'\bigl(\lambda_{Z}(D)\bigr)
\\
   -\lambda_{Z}(D_2)f_1^{(2)}\bigl(N(\mu, D_1, D_2)\bigr)\cdot N'(\mu, D_1, D_2) + \lambda_{Z}(D)f_1^{(2)}\bigl(N(\mu, D, D)\bigr)\cdot N'(\mu, D, D)
\\
    +\gamma_{1}f_{1}'\bigl(N(\mu, D_1, D_2)\bigr)\cdot N'(\mu, D_1, D_2)- \gamma_{1}f_{1}'\bigl(N(\mu, D, D)\bigr)\cdot N'(\mu, D, D).
\end{multline}  
For \eqref{p2primebound}, the same sort of building block approach on terms in \eqref{p2primedifference} using 
bounds 5 and 8  from proposition \ref{constituentbounds}, then bounds 1, 6, and 7, 
and finally bounds 3 and 7 delivers the bound \eqref{p2primebound}.
\end{proof}
Now we embark on the proof of proposition \ref{constituentbounds}.  Part of this work is made easier by the fact that the
rate functions in \eqref{NPZsys} are explicitly linear in $D$, $D_1$, and $D_2$.  On the other hand, because other quantities
such as $N$, $P$, and $Z$
depend implicitly on  $D$, $D_1$, and $D_2$, the details of the analyses are somewhat lengthy.
\begin{proof}[Proof of bound 1]
 By definition and the mean value theorem applied to $f_2$,  we have
  \begin{equation*}
   (D_2 - D)/\gamma_2 = f_2( \lambda_Z(D_2)\bigr) - f_2\bigl(\lambda_Z(D)\bigr) = f_2'(P_1)\cdot\bigl(\lambda_Z(D_2) - \lambda_Z(D)\bigr)
  \end{equation*}
for some number $P_1$ between $\lambda_Z(D_2)$ and $\lambda_Z(D)$. Consequently, 
\begin{equation*}
\abs{\lambda_Z(D_2) - \lambda_Z(D)} = \frac{\abs{D_2-D}}{\gamma_2\cdot f_2'(P_1)} 
   \leq \frac{\Dist{(D_1, D_2)}{(D, D)}}{\gamma_2\cdot f_2'(P_1)}. 
\end{equation*}
Since $P_1$ is also close to $\lambda_Z(D)$,  we may assume by lemma \ref{f2primebound} that
$f_2'(P_1) > f_2'\bigl(\lambda_Z(D)\bigr)/2$. 
Thus, for $D_2$ sufficiently close to $D$
\begin{equation*}
  \abs{\lambda_Z(D_2) - \lambda_Z(D)}
   \leq \frac{\Dist{(D_1, D_2)}{(D, D)}}{\gamma_2\cdot f_2'(P_1)}
< \frac{2 \cdot \Dist{(D_1, D_2)}{(D, D)}}{\gamma_2\cdot f_2'\bigl(\lambda_Z(D)\bigr)} \qedhere
\end{equation*}
\end{proof}
\begin{proof}
  [Proof of bound 2] Fix $\mu$ in the interval $I$.
We may use the mean-value theorem  \cite[p.103, Corollary~1]{LangAnalysisII} for functions of variables $(D_1, D_2)$, obtaining
\begin{equation*}
  \bigabs{f_1\bigl(N(\mu, D_1, D_2)\bigr) - f_1\bigl(N(\mu, D, D)\bigr)} 
              \leq \bignorm{\nabla (f_1 \com N)(\widehat{D}_1, \widehat{D}_2)}\cdot \Dist{(D_1, D_2)}{(D, D)},
\end{equation*}
where $(\widehat{D}_1, \widehat{D}_2)$ is a point on the line segment connecting $(D_1, D_2)$ and $(D,D)$. Therefore, we have to bound
the magnitude of the gradient $\norm{\nabla (f_1 \com N)(\widehat{D}_1, \widehat{D}_2)}$ in a disc surrounding $(D,D)$ by a constant.

To obtain information about the partial derivatives 
$\partial(f_1 \com N)/\partial D_1 = f_1'(N)\cdot (\partial N / \partial D_1)$ 
and 
$\partial(f_1 \com N)/\partial D_2 = f_1'(N)\cdot (\partial N / \partial D_2)$,
we return to the defining equation 
\begin{equation*}
   0 =G_1(N, Z, \mu, D_1, D_2) = D(\mu - N ) - f_1(N)\lambda_Z(D_2)
\end{equation*}
  and differentiate with respect to $D_1$ and $D_2$. We obtain
  \begin{align*}
    0 &= \partial G_1/ \partial D_1 
         = \partial G_1/\partial N \cdot \partial N/ \partial D_1 
       \\
    &= \bigl( -D  - f_1'(N) \cdot \lambda_Z(D_2)\bigr) \cdot (\partial N / \partial D_1) 
\intertext{and}
0 &= \partial G_1/ \partial D_2 
         = \partial \bigl(D(\mu{-}N )\bigr) / \partial D_2  - \partial \bigl(  f_1(N)\lambda_Z(D_2)\bigr)/\partial D_2
 \\
  &= -D \cdot  (\partial N / \partial D_2) - \bigl( f_1'(N)\cdot (\partial N /\partial D_2) \cdot \lambda_Z(D_2) + f_1(N) \lambda_Z'(D_2)\bigr).
  \end{align*}
From the first of these equations 
\begin{align}
  \frac{\partial N}{\partial D_1}(\mu, D_1, D_2) &= 0, \label{PND1}
\intertext{since $D + f_1'\bigl(N(\mu, D_1, D_2)\bigr)\cdot \lambda_Z(D_2) > 0$, and from the second}
\frac{\partial N}{\partial D_2}(\mu, D_1, D_2) 
      &= - \frac{f_1\bigl(N(\mu, D_1, D_2)\bigr) \lambda_Z'(D_2)}{D + f_1'\bigl(N(\mu, D_1, D_2)\bigr)\cdot \lambda_Z(D_2)}\label{PND2}
\end{align}
Now we obtain a bound on the gradient via
\begin{align} 
\biggabs{\frac{\partial N}{\partial D_2}(\mu, D_1, D_2)} 
            &= \frac{f_1\bigl(N(\mu, D_1, D_2)\bigr) \lambda_Z'(D_2)}{D + f_1'\bigl(N(\mu, D_1, D_2)\bigr)\cdot \lambda_Z(D_2)}
\\ 
&< \frac{f_1\bigl((N(\mu, D_1, D_2)\bigr) \lambda_Z'(D_2)}{D} 
          = \frac{f_1\bigl((N(\mu, D_1, D_2)\bigr)}{D \cdot \gamma_2 \cdot f_2'\bigl(\lambda_Z(D_2)\bigr)},  \notag
\intertext{since $f_1'\bigl(N(\mu, D_1, D_2)\bigr)\cdot \lambda_Z(D_2) > 0$ and because the defining relation
$  f_2\bigl(\lambda_Z(D_2)\bigr) = D_2/\gamma_2$ implies $\lambda_Z'(D_2) = 1/\bigl( \gamma_2 \cdot f_2'\bigl(\lambda_Z(D_2)\bigr) \bigr)$,}
&< \frac{2\cdot f_1\bigl((N(\mu, D_1, D_2)\bigr)}{D \cdot \gamma_2 \cdot f_2'\bigl(\lambda_Z(D)\bigr)},  \label{PND2bound}
\end{align}
since  $f'_2\bigl(\lambda_Z(D_2)\bigr) > f_2'\bigl(\lambda_Z(D)\bigr)/2$ by choice of $\Delta$ and lemma \ref{f2primebound}.
Also, $f_1\bigl((N(\mu, D_1, D_2)\bigr)$ is bounded by
$\lim_{N \rightarrow \infty}f_1(N)$, since $N(\mu, D_1,D_2)$ is unbounded as $\mu$ tends to infinity according to theorem \ref{boundedness of Z_2}. 
Thus, $\abs{\partial N/\partial D_2(\mu, D_1, D_2)}$ is bounded by a constant in the disc $\Delta$. 

We also observe that $f_1'\bigl(N(\mu, \widehat{D}_1, \widehat{D}_2)\bigr)$ is bounded by a constant depending only on $I{\times}\Delta$,
because of  the convexity of the closed disc $\Delta$ centered at $(D,D)$ from which
we choose $(D_1, D_2)$.  

Combining all this information, $\norm{\nabla (f_1 \com N)(\widehat{D}_1, \widehat{D}_2)}$ is bounded by a constant depending
on $I{\times}\Delta$. Therefore, $\bigabs{f_1\bigl(N(\mu, D_1, D_2)\bigr) - f_1\bigl(N(\mu, D, D)\bigr)} $ is bounded by a constant times
$\Dist{(D_1, D_2)}{(D, D)}$.
\end{proof}
\begin{proof}
  [Proof of bound 3]
We again use the mean-value theorem  \cite[p.103, Corollary~1]{LangAnalysisII} for functions of variables $(D_1, D_2)$, obtaining
\begin{equation*}
  \bigabs{f'_1\bigl(N(\mu, D_1, D_2)\bigr) - f'_1\bigl(N(\mu, D, D)\bigr)} 
       \leq \bignorm{\nabla (f'_1 \com N)(\widehat{D}_1, \widehat{D}_2)}\cdot \Dist{(D_1, D_2)}{(D, D)},
\end{equation*}
where $(\widehat{D}_1, \widehat{D}_2)$ is a point on the line segment connecting $(D_1, D_2)$ and $(D,D)$. Therefore, we have to bound
the magnitude of the gradient $\norm{\nabla (f'_1 \com N)(\widehat{D}_1, \widehat{D}_2)}$ in a disc surrounding $(D,D)$.

To obtain information about the partial derivatives 
$\partial(f_1' \com N)/\partial D_1 = f_1^{(2)}(N)\cdot (\partial N / \partial D_1)$ 
and 
$\partial(f_1' \com N)/\partial D_2 = f_1^{(2)}(N)\cdot (\partial N / \partial D_2)$,
we cite \eqref{PND1} for the vanishing of $\partial N/\partial D_1$ 
and the bound on $\partial N/\partial D_2$ obtained in \eqref{PND2bound}.
We also observe that $f_1^{(2)}\bigl(N(\mu, \widehat{D}_1, \widehat{D}_2)\bigr)$ is bounded by a constant, assuming a continuous second
derivative of $f_1$.
The constant depends on the $\mu$-interval $I$
and the closed disc $\Delta$ centered at $(D,D)$ from which
we choose $(D_1, D_2)$, plus the convexity of the disc.

Combining all this information, $\norm{\nabla (f_1' \com N)(\widehat{D}_1, \widehat{D}_2)}$ is bounded by a constant depending
on $I{\times}\Delta$. Therefore, $\bigabs{f_1'\bigl(N(\mu, D_1, D_2)\bigr) - f_1'\bigl(N(\mu, D, D)\bigr)} $ is bounded by a constant times
$\Dist{(D_1, D_2)}{(D, D)}$.
\end{proof}
\begin{proof}
  [Proof of bound 4]
Since $\mu$ is a fixed number in $I$, we use again the mean value theorem for functions of $(D_1, D_2)$.
\begin{equation*}
  \abs{Z(\mu, D_1, D_2) - Z(\mu, D, D)} \leq \bignorm{\nabla Z(\widehat{D}_1, \widehat{D}_2)}\cdot \Dist{(D_1, D_2)}{(D, D)},
\end{equation*}
where $(\widehat{D}_1, \widehat{D}_2)$ is a point on the line segment connecting $(D_1, D_2)$ and $(D,D)$. Therefore, we have to bound
the magnitude of the gradient $\norm{\nabla Z(\widehat{D}_1, \widehat{D}_2)}$ in a disc surrounding $(D,D)$.

To obtain information about $\partial Z/\partial D_1$ and $\partial Z/ \partial D_2$ we return to the defining equation
\begin{equation*}
 0 = G_2(N, Z, \mu, D_1, D_2)  = \gamma_1f_1(N)\lambda_Z(D_2) - D_1\lambda_Z(D_2) - (D_2/\gamma_2) Z.
\end{equation*}
Differentiating with respect to $D_1$, we obtain
\begin{align*}
  0 &= \partial G_2/\partial D_1
\\  & = \partial \bigl(\gamma_1f_1(N)\lambda_Z(D_2) \bigr)/ \partial D_1 
                                 - \partial \bigl( (D_2/\gamma_2) Z \bigr)/ \partial D_1
                                 - \partial \bigl(  D_1\lambda_Z(D_2) \bigr)/\partial D_1
\\
    &= \gamma_1f_1'(N)\lambda_Z(D_2)\cdot \partial N/ \partial D_1 
       - (D_2/\gamma_2) \partial Z / \partial D_1 
      - \lambda_Z(D_2)
\\  &=  - (D_2/\gamma_2) \partial Z / \partial D_1 
      - \lambda_Z(D_2),
\end{align*}
since $\partial N/\partial D_1 = 0$ by equation \eqref{PND1}.  
Differentiating with respect to $D_2$, we obtain
\begin{align*}
  0 &= \partial G_2/\partial D_2
\\
    & = \partial \bigl(\gamma_1f_1(N)\lambda_Z(D_2) \bigr)/ \partial D_2 
                                 - \partial \bigl( (D_2/\gamma_2) Z \bigr)/\partial D_2
                                 - \partial \bigl(  D_1\lambda_Z(D_2) \bigr)/\partial D_2
\\
  &=\gamma_1\bigl[f_1'(N)\lambda_Z(D_2)\cdot \partial N/ \partial D_2  {+}f_1(N)\lambda_Z'(D_2)\bigr]
   - \bigl[(D_2/\gamma_2) \partial Z / \partial D_2 {+} Z/\gamma_2\bigr]
   - D_1\lambda_Z'(D_2).
\end{align*}
Rewriting these equations, we obtain
\begin{align}
  \frac{\partial Z}{\partial D_1}(\mu, D_1, D_2)  &=  - \frac{\lambda_Z(D_2)\cdot \gamma_2}{D_2}  \label{PZD1}
\intertext{and}
  \frac{\partial Z}{\partial D_2}(\mu, D_1, D_2)  
    &= \frac{\gamma_1\gamma_2}{D_2}f_1'\bigl(N(\mu, D_1, D_2)\bigr)\cdot \frac{\partial N}{\partial D_2}(\mu, D_1, D_2)\cdot \lambda_Z(D_2)   \notag
\\   &+ \frac{\gamma_2}{D_2}\cdot\bigl( \gamma_1 f_1\bigl( N(\mu, D_1, D_2) \bigr) -D_1\bigr)\cdot \lambda_Z'(D_2)
       - \frac{Z(\mu, D_1, D_2)}{D_2}.  \label{PZD2}
\end{align}
To bound $\partial Z/\partial D_1$, we require a bound on $\lambda_Z(D_2)$. By definition $f_2\bigl(\lambda_Z(D_2)\bigr) = D_2/\gamma_2$, so
 restricting $D_2$ to be close to $D$ prevents $D_2/\gamma_2$ from approaching $\lim_{P \rightarrow \infty} f_2(P)$.
Consequently, $\lambda_Z(D_2)$ is a bounded distance from $\lambda_Z(D)$. 

To bound $\partial Z/ \partial D_2$, we discuss the terms on the righthand side of \eqref{PZD2} in reverse order.  
The term $D_2^{-1}Z(\mu, D_1, D_2)$ is bounded if $(D_1,D_2)$ is close to $(D,D)$, for it will be close to $D_2^{-1}Z(\mu, D,D)$. 
In turn $Z(\mu, D,D)$ is bounded by $\lim_{\mu \rightarrow \infty}Z(\mu, D,D)$, which exists by theorem \ref{boundedness of Z_2}.
Concerning the second term, $f_2'\bigl(\lambda_Z(D_2)\bigr)\cdot \lambda_Z'(D_2) = 1/\gamma_2$, so 
\begin{multline*}
   \frac{\gamma_2}{D_2}\cdot\bigl( \gamma_1 f_1\bigl( N(\mu, D_1, D_2) \bigr) -D_1\bigr)\cdot \lambda_Z'(D_2)
\\
=   \frac{\gamma_2}{D_2}\cdot\bigl( \gamma_1 f_1\bigl( N(\mu, D_1, D_2) \bigr) -D_1\bigr)\cdot \frac{1}{\gamma_2\cdot f_2'\bigl(\lambda_Z(D_2)\bigr)}
=    \frac{\bigl( \gamma_1 f_1\bigl( N(\mu, D_1, D_2) \bigr) -D_1\bigr)}{D_2\cdot f_2'\bigl(\lambda_Z(D_2)\bigr)},
\end{multline*}
where $f_1\bigl(N(\mu, D_1, D_2)\bigr)$ can be bounded in terms of $\lim_{N \rightarrow \infty}f_1(N)$, and
$f_2'\bigl(\lambda_Z(D_2)\bigr)$ is bounded away from zero by lemma \ref{f2primebound}.
Concerning the first term, substitute the expression for $\partial N/\partial D_2$ given in \eqref{PND2}, obtaining
\begin{multline*}
  \frac{\gamma_1\gamma_2}{D_2}f_1'\bigl(N(\mu, D_1, D_2)\bigr)\frac{\partial N}{\partial D_2}(\mu, D_1, D_2)\cdot \lambda_Z(D_2)
\\
  = -\frac{\gamma_1\gamma_2}{D_2}f_1'\bigl(N(\mu, D_1, D_2)\bigr)\cdot
     \frac{f_1\bigl(N(\mu, D_1, D_2)\bigr) \lambda_Z'(D_2)}{D + f_1'\bigl(N(\mu, D_1, D_2)\bigr)\cdot \lambda_Z(D_2)}\cdot \lambda_Z(D_2)
\\
  = -\frac{\gamma_2 \gamma_1}{D_2}\cdot
    \frac{ f_1\bigl(N(\mu, D_1, D_2)\bigr) \cdot \lambda_Z'(D_2)}{\Bigl(\frac{D}{\lambda_Z(D_2)\cdot f_1'\bigl(N(\mu, D_1, D_2)\bigr)} + 1\Bigr)}
\end{multline*}
Consequently, 
\begin{multline*}
  \Bigabs{\frac{\gamma_1\gamma_2}{D_2}f_1'\bigl(N(\mu, D_1, D_2)\bigr)\frac{\partial N}{\partial D_2}(\mu, D_1, D_2)\cdot \lambda_Z(D_2)}
\\ 
\leq \frac{\gamma_2 \gamma_1}{D_2} \cdot f_1\bigl(N(\mu, D_1, D_2)\bigr) \cdot \lambda_Z'(D_2)=\frac{\gamma_1 f_1\bigl(N(\mu, D_1, D_2)\bigr)}{D_2 f_2'\bigl(\lambda_Z(D_2)\bigr)},
\end{multline*}
where we use again the fact that $\lambda_Z'(D_2) = 1/\gamma_2 f_2'\bigl(\lambda_Z(D_2)\bigr)$. Arguing as above, we conclude this term 
 can be bounded by a constant, and, therefore, $\abs{\partial Z/\partial D_2(\mu, D_1, D_2)}$ itself is bounded by a constant
depending only on the domain $I{\times}\Delta$.

Combining these bounds $\norm{\nabla Z(\widehat{D}_1, \widehat{D}_2)}$ is bounded by a constant, so we conclude
that $\abs{Z(\mu, D_1, D_2) - Z(\mu, D, D)}$ is bounded by a constant times $\Dist{(D_1, D_2)}{(D, D)}$.
\end{proof}
\begin{proof}
  [Proof of bound 5]
By the mean value theorem
\begin{equation*}
  f_2'\bigl(\lambda_Z(D_2)\bigr) - f_2'\bigl(\lambda_Z(D)\bigr) = f_2^{(2)}(P_2)\cdot\bigl(\lambda_Z(D_2) - \lambda_Z(D)\bigr)
\end{equation*}
for some $P_2$ between $\lambda_Z(D_2)$ and $\lambda_Z(D)$.  Moreover, 
\begin{equation*}
  D_2/\gamma_2 - D/ \gamma_2 = f_2\bigl(\lambda_Z(D_2)\bigr) - f_2\bigl(\lambda_Z(D)\bigr) = f_2(P_1')\bigl(\lambda_Z(D_2) - \lambda_Z(D)\bigr),
\end{equation*}
for some $P_1'$ between $\lambda_Z(D_2)$ and $\lambda_Z(D)$.  We may combine to obtain
\begin{equation*}
   f_2'\bigl(\lambda_Z(D_2)\bigr) - f_2'\bigl(\lambda_Z(D)\bigr) = \frac{f_2^{(2)}(P_2)}{\gamma_2 f_2'(P_1)}\cdot(D_2 - D).
\end{equation*}
Since we have control of the continuous derivatives $f_2'$ and $f_2^{(2)}$ on the interval $J$ around $\lambda_Z(D)$, the difference
$\abs{ f_2'\bigl(\lambda_Z(D_2)\bigr) - f_2'\bigl(\lambda_Z(D)\bigr)}$ is indeed bounded by a constant times $\Dist{(D_1, D_2)}{(D, D)}$.
\end{proof}
\begin{proof}
  [Proof of bound 6]  This is precisely parallel to the proofs of bounds 2 and 3.
We may use the mean-value theorem for functions of variables $(D_1, D_2)$, obtaining
\begin{equation*}
  \bigabs{f_1^{(2)}\bigl(N(\mu, D_1, D_2)\bigr) - f_1^{(2)}\bigl(N(\mu, D, D)\bigr)} 
          \leq \bignorm{\nabla (f_1^{(2)} \com N)(\widehat{D}_1, \widehat{D}_2)}\cdot \Dist{(D_1, D_2)}{(D, D)},
\end{equation*}
where $(\widehat{D}_1, \widehat{D}_2)$ is a point on the line segment connecting $(D_1, D_2)$ and $(D,D)$. Therefore, we have to bound
the magnitude of the gradient $\norm{\nabla (f_1^{(2)} \com N)(\widehat{D}_1, \widehat{D}_2)}$ in a disc surrounding $(D,D)$.

To bound the partial derivatives
\begin{equation*}
  \partial(f_1^{(2)} \com N)/\partial D_1 = f_1^{(3)}(N)\cdot (\partial N / \partial D_1) 
\; \text{and} \;  
\partial(f_1^{(2)} \com N)/\partial D_2 = f_1^{(3)}(N)\cdot (\partial N / \partial D_2),
\end{equation*}
we have the vanishing of $\partial N/\partial D_1$ by \eqref{PND1} 
and a bound on $\abs{\partial N/\partial D_2}$ from \eqref{PND2bound}.
We also observe that $f_1^{(3)}\bigl(N(\mu, \widehat{D}_1, \widehat{D}_2)\bigr)$ is bounded by a constant, assuming a continuous third
derivative of $f_1$.

Combining all this information, $\norm{\nabla (f_1^{(2)} \com N)(\widehat{D}_1, \widehat{D}_2)}$ is bounded by a constant depending
on $I{\times}\Delta$. Therefore, $\abs{f_1^{(2)}\bigl(N(\mu, D_1, D_2)\bigr) - f_1^{(2)}\bigl(N(\mu, D, D)\bigr)} $ is bounded by a constant times
$\Dist{(D_1, D_2)}{(D, D)}$.
\end{proof}
\begin{proof}
  [Proof of bound 7]
For this proof, return to the defining relation for $N(\mu, D_1, D_2)$, namely,
\begin{equation*}
   0 =G_1(N, Z, \mu, D_1, D_2) = D(\mu - N ) - f_1(N)\lambda_Z(D_2),
\end{equation*}
and differentiate with respect to $\mu$, obtaining
\begin{equation*}
 0 = D + \bigl(-D - f_1'\bigl(N(\mu, D_1, D_2)\bigr) \cdot \lambda_Z(D_2) \bigr) \cdot N'(\mu, D_1, D_2),
\end{equation*}
so
\begin{equation*}
  N'(\mu, D_1, D_2) = \frac{D}{D + f_1'\bigl(N(\mu, D_1, D_2)\bigr)\cdot \lambda_Z(D_2)}.
\end{equation*}
Consequently, 
\begin{multline}
  \abs{N'(\mu, D_1, D_2) - N'(\mu, D, D) } 
\\
= \Bigabs{\frac{D}{D + f_1'\bigl(N(\mu, D_1, D_2)\bigr)\cdot \lambda_Z(D_2)}
                                   - \frac{D}{D + f_1'\bigl(N(\mu, D, D)\bigr)\cdot \lambda_Z(D)} } 
\\
    = \Bigabs{\frac{D f_1'\bigl(N(\mu, D, D)\bigr)\cdot \lambda_Z(D) - D f_1'\bigl(N(\mu, D_1, D_2)\bigr)\cdot \lambda_Z(D_2)}
       {\bigl(D + f_1'\bigl(N(\mu, D_1, D_2)\bigr)\cdot \lambda_Z(D_2)\bigr)\cdot\bigl( D + f_1'\bigl(N(\mu, D, D)\bigr)\cdot \lambda_Z(D)\bigr)} } 
\\  
\leq \Bigabs{\frac{f_1'\bigl(N(\mu, D, D)\bigr)\cdot \lambda_Z(D) - f_1'\bigl(N(\mu, D_1, D_2)\bigr)\cdot \lambda_Z(D_2)}{D}}.
\label{Nprimedifference}
\end{multline}
Applying lemma \ref{elementary} to \eqref{Nprimedifference} with bounds 1 and 3 as input, 
we find that $\abs{N'(\mu, D_1, D_2) - N'(\mu, D, D) }$
is bounded by a constant times $\Dist{(D_1, D_2)}{(D, D)}$.
\end{proof}
\begin{proof}
  [Proof of bound 8]
For this proof, return to the defining relation for $Z$, namely, 
\begin{equation*}
  0 = G_2(N, Z, \mu, D_1, D_2) = \gamma_1f_1(N)\lambda_Z(D_2) - D_1\lambda_Z(D_2) - (D_2/\gamma_2) Z,
\end{equation*}
and differentiate with respect to $\mu$, obtaining
\begin{equation*}
  0 = \gamma_1f_1'\bigl(N(\mu, D_1, D_2)\bigr)\cdot \lambda_Z(D_2)\cdot N'(\mu, D_1, D_2) - (D_2/\gamma_2)\cdot Z'(\mu, D_1, D_2).
\end{equation*}
Thus,
\begin{equation*}
  Z'(\mu, D_1, D_2) = \frac{\gamma_1\gamma_2}{D_2} \cdot f_1'\bigl(N(\mu, D_1, D_2)\bigr)\cdot \lambda_Z(D_2)\cdot N'(\mu, D_1, D_2),
\end{equation*}
and
\begin{equation}  \begin{split}
  Z'(\mu, D_1, D_2) &- Z'(\mu, D, D) 
\\
           &= \frac{\gamma_1\gamma_2}{D_2} \cdot f_1'\bigl(N(\mu, D_1, D_2)\bigr)\cdot \lambda_Z(D_2)\cdot N'(\mu, D_1, D_2)
\\           &\quad \quad  - \frac{\gamma_1\gamma_2}{D} \cdot f_1'\bigl(N(\mu, D, D)\bigr)\cdot \lambda_Z(D)\cdot N'(\mu, D, D)
\\
  &= \frac{\gamma_1\gamma_2}{D_2 D}\cdot\bigl[ D f_1'\bigl(N(\mu, D_1, D_2)\bigr)\cdot \lambda_Z(D_2)\cdot N'(\mu, D_1, D_2)
\\           &\hspace{10em}         - D_2 f_1'\bigl(N(\mu, D, D)\bigr)\cdot \lambda_Z(D)\cdot N'(\mu, D, D) \bigr]
\end{split}  \end{equation}
Obviously $\abs{D - D_2} \leq \Dist{(D_1, D_2)}{(D, D)}$, so we make several applications of lemma \ref{elementary} 
to combine this fact with bounds 1, 3, and 7 to deduce that
$\abs{ Z'(\mu, D_1, D_2) - Z'(\mu, D, D)  } $ is bounded by a constant times $\Dist{(D_1, D_2)}{(D,D)}$.
\end{proof}
\bibliographystyle{plain}
\bibliography{bjs_npp_bifurcations}
\end{document}